\documentclass[letterpaper,11pt,leqno]{amsart} 
\usepackage[T1]{fontenc}
\usepackage[T1]{fontenc}
\usepackage[utf8]{inputenc}
\usepackage{bera}
\usepackage[portrait,margin=3cm]{geometry} 
\usepackage[mathscr]{euscript}
\usepackage[colorlinks=true,linkcolor=blue,citecolor=blue,urlcolor=blue]{hyperref} 
\usepackage{amsmath,amssymb,amsthm,amsfonts,amsbsy,latexsym,dsfont,color,graphicx,enumitem}
\usepackage[numeric,initials,nobysame]{amsrefs} 

\usepackage{float}
\floatstyle{ruled}
\newfloat{algorithm}{tbp}{loa}
\providecommand{\algorithmname}{Algorithm}
\floatname{algorithm}{\protect\algorithmname}

\usepackage{enumitem}
\usepackage{color,comment}

\newcommand{\R}{\mathbb{R}}

\newtheorem{thm}{Theorem}[section]

\newtheorem{lem}[thm]{Lemma}

\newtheorem{remark}[thm]{Remark}

\newtheorem{prop}[thm]{Proposition}
\newtheorem{conjecture}[thm]{Conjecture}

\theoremstyle{definition}

\newtheorem*{claim*}{Claim}

\newtheorem{rem}[thm]{Remark}

\numberwithin{equation}{section}
\newtheorem*{acknowledgement}{Acknowledgement}

\begin{document}

\title[]{
On a conjecture of a P\'olya functional for triangles and rectangles}

\author[Banuelos]{Rodrigo Ba\~{n}uelos{$^{\dag}$}}
\thanks{\footnotemark {$\dag$} Research was supported in part by NSF Grant DMS-1854709}
\address{ Department of Mathematics\\
Purdue University\\
West Lafayette, IN 47907,  U.S.A.} \email{banuelos@purdue.edu}

\author[Mariano]{Phanuel Mariano{$^\ddag$}}
\thanks{\footnotemark {$\ddag$} Research was supported in part by  NSF Grant DMS-2316968 and a 2023-2024 Union College Faculty Research Fund Award}
	\address{Department of Mathematics\\
		Union College\\
		Schenectady, NY 12308,  U.S.A.}
	\email{marianop@union.edu}



\begin{abstract}  
We consider the  functional given by the product of the first Dirichlet eigenvalue  and the torsional rigidity of planar domains normalized by
the area. This scale invariant functional was studied by P\'olya and Szeg\H{o} in 1951 who showed that it is bounded above by 1 for all domains.  It has been conjectured that within the class of bounded convex planar domains 
the functional is bounded below by $\pi^{2}/24$ and above by $\pi^{2}/12$ and that these bounds are sharp. Remarkably, the conjecture remains open even within the class of triangles.  The purpose of this paper is to prove the conjecture in this case.
The conjecture is also proved for rectangles where a  stronger monotonicity property is verified.  Finally, the upper bound also holds for tangential quadrilateral.

\end{abstract}

\maketitle

\tableofcontents

\section{Introduction and Main Results}

Consider  an open connected set $D\subset \R^d$, $d\geq 2$, which we refer to as a domain. Further,  assume the Lebesgue measure of $D$, denoted by $|D|$, is finite.  
The torsion function $u_{D}$ is
the unique weak solution to the  boundary value problem 
\begin{align}\label{TorionFunc}
\begin{cases}
\Delta u_{D}=-1 & \text{in } D,\\
u_D\in H_0^1\left(D\right). & 
\end{cases}
\end{align}
It is well known that  $\left\Vert u_{D}\right\Vert _{\infty}<\infty$, $u_D\geq 0$ and $u_D\in C^\infty\left(D\right)$. In fact, $u_D(x)$ satisfies the isoperimetric inequality $u_D(x)\leq u_{D^*}(0)$, where $D^*$, is the ball centered at the origin with $|D^*|=|D|$.  It is also a well-known (and widely-used) fact that 
$u_D(x)=\frac{1}{2}\mathbb{E}_{x}\left[\tau_{D}\right]$, where the right hand side is the expectation of the first exit time $\tau_{D}$ of Brownian motion from the domain $D$ starting at the point $x\in D$.  Although this probabilisitic interpretation is very useful in many ways, it will not be explicitly used in this paper other than from time to time to observe domain monotonicity of various quantities.

The {\bf torsional rigidity} $T(D)$ of $D$ is defined by 
\[
T(D)=\int_{D}u_{D}(x)dx.
\]The torsional rigidity $T(D)$ has been studied and applied extensively in the theory of elasticity \cite{Timoshenko-Goodier}. The torsional rigidity $T(D)$ is related to the computation that measures the resistance of a beam with cross-sections   $D$ to twisting forces. Probabilistically, the quantity $T\left(D\right)/\left|D\right|$ can be written as $\frac{1}{2}\mathbb{E}_{\mu}\left[\tau_{D}\right]$
which is the mean exit time of Brownian motion started in $D$ whose
starting point is averaged by the uniform distribution $\mu$ on $D$.

Let  $\lambda_{1}\left(D\right)$ be the \textbf{first Dirichlet eigenvalue} of  $-\Delta_{D}$. In \cite{Polya-1948b}, P\'olya showed that the process of Steiner symmetrization decreases $\lambda_{1}\left(D\right)$ while increasing  $T\left(D\right)$. In this paper, we study the relation between $\lambda_{1}\left(D\right)$ and $T\left(D\right)$ through the following functional
\begin{equation}\label{defn:F-functional}
F(D)=\frac{\lambda_1(D)T(D)}{\left|D\right|},
\end{equation}
which we refer to by the P\'olya functional following \cite{Vandenberg-Ferone-Nitsch-Trombetti-2019a}. This functional was studied by  P\'olya-Szeg\H{o} \cite[p. 91]{Polya-Szego-1951} in 1951 who showed that $F(D)\leq1$. This was known to P\'olya as early as 1947 in \cite[Eq. (2)]{Polya-1947a}. By a result in \cite[Theorem 1.2]{Vandenberg-Ferrone-Nitsch-Trombetti-2016}, this bound is sharp over all open connected sets in $\mathbb{R}^d$. The problem of obtaining sharp upper and lower bounds on this functional and its extremals for subclasses of domains has been extensively investigated for many years, and especially in the last ten years or so. Among the class of bounded convex domains in the plane the following conjecture is open.

\begin{conjecture}[Conjecture 4.2 in \cite{Berg-Buttazzo-Pratelli-2021}, see also \cite{Berg-Buttazzo-Velichkov-2015,Vandenberg-Ferrone-Nitsch-Trombetti-2016}]\label{Conjecture-convex}
For all bounded convex planar domains $D\subset\mathbb{R}^2$,
\begin{equation}\label{Conjecture-convex-ineq}
\frac{\pi^{2}}{24} <F(D)<\frac{\pi^{2}}{12}, 
\end{equation}
and these bounds are sharp. The lower bound is attained for a collapsing
sequence of isosceles triangles converging down to an interval. The upper bound is
attained by a sequence of elongating rectangles approaching the infinite strip.  
\end{conjecture}

Remarkably, the conjecture remains open even within the class of triangles.  The purpose of this paper is to prove the conjecture in this case. We will also consider a \textbf{tangential quadrilateral}, which is any convex quadrilateral that contains an incircle that is tangent to all sides. Examples include kites which include rhombi.

\begin{thm}
\label{thm:MainResult} 
Suppose  $D\subset\mathbb{R}^{2}$ is a triangle or a rectangle. Then 
\begin{equation}
\frac{\pi^{2}}{24}< F(D)< \frac{\pi^{2}}{12}.\label{eq:Main-Theorem-Inequality}
\end{equation}
The upper bound is attained for a sequence of elongating rectangles
approaching an infinite strip. The lower bound is attained for any sequence
of triangles collapsing down to an interval. 

The upper bound  also holds for any tangential quadrilateral.   
\end{thm}
\begin{remark}\label{strongerrectangles}
A stronger monotonicity result is given for rectangles, where in Theorem \ref{thm:Main-Rec} it is shown that $F\left(R_{a,1}\right)$ is increasing for
$a\geq1$ for rectangles $R_{a,b}=\left(-a,a\right)\times\left(-b,b\right)$. 
\end{remark}

As already mentioned, while Conjecture \ref{Conjecture-convex} remains open for general convex domains, progress has been made for other smaller classes of domains. In \cite{Vandenberg-Ferone-Nitsch-Trombetti-2019a}, the authors proved that the lower bound of Conjecture \ref{Conjecture-convex} is true 
for all domains $D$ that are either isosceles triangles
or rhombi. Moreover, they show this inequality is sharp
for a limiting sequence of collapsing isosceles triangles or rhombi that converge to an interval. It has also been shown in \cite[Proposition 5.2]{Berg-Buttazzo-Pratelli-2021} or \cite[Theorem 4.4]{Briani-Buttazzo-Prinari-2022} that the asymptotic limit of $F$ for thinning sequences of convex domains are always between the conjectured bounds.

 In this paper we shall only be concerned with domains in the class $\mathcal{C}_2$ of planar bounded convex domains. Some improvements on the upper bound $F(D)\leq 1$ valid for all planar domains have been obtained for the class $\mathcal{C}_2$. For example, the bound $F(D)\leq1-\frac{1}{11560}\approx0.999913$  was given in  \cite{Vandenberg-Ferrone-Nitsch-Trombetti-2016}. The best bound to date for all $D\in \mathcal{C}_2$ is $F(D)\leq 0.996613$ given in \cite{Ftouhi-2022}. Recently in \cite{Berg-Bucur-2024}, it has been shown that there exists a $c<1$ such that $F(D)<c$ for all simply connected planar domains. Improved lower bounds for $F\left(D\right)$ for all $D\in \mathcal{C}_2$ have also been obtained. In particular, it was shown in \cite{Vandenberg-Ferrone-Nitsch-Trombetti-2016} that $F\left(D\right)\geq \frac{\pi^2}{48}$ on $\mathcal{C}_2$. This has been improved to $F\left(D\right) \geq \frac{\pi^2}{32}$ for planar convex domains in $\mathcal{C}_2$ (see \cite[Prop. 3.2]{Briani-Buttazzo-Prinari-2022} and \cite[Remark 4.1]{Brasco-Mazzoleni-2020}).

In general, other than a few special cases,  there are no explicit formulas for the torsional rigidity or the first Dirichlet eigenvalue of triangles or general polygons. This makes proving sharp inequalities involving both $\lambda_1(D)$ and $T(D)$  difficult even for triangles. Despite this, 
sharp inequalities for the first Dirichlet eigenvalue of triangles, quadrilaterals and other polygons have been extensively investigated in the literature.  We point to the works of \cite{Freitas-2007,Freitas-2006,Antunes-Freitas-2006,Freitas-Laugesen-2021,Antunes-Freitas-2011,Endo-Liu-2023,Solynin-Zalgaller-2010,Arbon-etall-2022,Arbon-2022,Laugesen-etall-2017,Siudeja-2016,Indrei-2024} for some of this literature.
In particular, the inequalities and methods from  \cite{Freitas-Siudeja-2010,Siudeja-2007,Laugesen-Siudeja-2011,Siudeja-IU-2010} will be useful in the proof of our main result for various cases. Although not as extensive, there is also a sizable literature  
regarding the torsional rigidity of polygons; 
see \cite{Rolling-2023,Solynin-2020,Fleeman-Simanek-2019,Solynin-Zalgaller-2010,Benson-Laugesen-Minion-Siudeja-2016,Timoshenko-Goodier}. Although not directly related, it is interesting to note that other difficult spectral theory problems have also been studied for triangles. One such example is the well known Hot Spots Conjecture regarding the maximum of the Neumann eigenfunction corresponding to the first positive eigenvalue that was settled recently in \cite{Judge-Sugata-2019} for triangles, with earlier and recent contributions given by several authors  \cite{Siudeja-HotSpots-2015,Chen-Changfeng-Ruofei-2023,Banuelos-Burdzy-1999}; see also the Polymath Project 7 \cite{Polymath}. That conjecture remains open for general convex domains. Inverse spectral problems have also been considered for triangles such as in  \cite{Gomez-Serrano-Orriols-2021,Meyerson-McDonald-2017} and higher $L^1$-moment spectrum bounds have been studied in \cite{Dryden-Langford-McDonald-2017}.

We now discuss other functionals where their sharp bounds would imply sharp bounds for $F$. Consider the functional given by 
\[
\Psi\left(D\right)=\frac{T\left(D\right)}{\left|D\right|R_{D}^{2}},
\]
where $R_D$ is the inradius. By the results in \cite{Brasco-Mazzoleni-2020,Polya-Szego-1951} it follows that $\Psi\left(D\right)\geq\frac{1}{8}$
for $\mathcal{C}_{2}$. Combining this with the Hersch-Protter inequality
$\lambda_{1}\left(D\right)R_{D}^{2}\geq\frac{\pi^{2}}{4}$ gives the
best known bound of $F\left(D\right)\geq\frac{\pi^{2}}{32}$, as mentioned in \cite[Remark 4.1]{Brasco-Mazzoleni-2020}.

There is another functional whose lower bounds imply lower bounds for $F$. The mean-to-max ratio of the torsion function (also referred to as the ``efficiency")
is defined by 
\[
\Phi\left(D\right)=\frac{T\left(D\right)}{\left|D\right|M\left(D\right)},
\]
where $M\left(D\right)$ is the maximum of the torsion function $u_{D}$. Various authors have proved upper and lower bounds for $\Phi$ over
convex domains for more general operators; see \cite{DellaPietra-Gavitone-Guarino-2018,Bueno-Ercole-2011,Henrot-Lucardesi-Philippin-2018,Briani-Bucur-2023}. The best lower bound so far is $\Phi\left(D\right)\geq\frac{1}{4}$
given in \cite{DellaPietra-Gavitone-Guarino-2018}. A result by Payne in \cite{Payne-1981} shows that $\lambda_{1}\left(D\right)M\left(D\right)\geq\frac{\pi^{2}}{8}$.
Combining these two bounds 
implies that $F\left(D\right)\geq\frac{\pi^{2}}{32}$. It is conjectured in \cite{Henrot-Lucardesi-Philippin-2018} 
that the bound $\Phi\left(D\right)\geq\frac{1}{3}$ holds  for convex planar domains which
would also imply the conjecture on $F\left(D\right)$.

\subsection{Discussion of method of proof}
As mentioned before, there are no explicit formulas for the first Dirichlet eigenvalue nor for the torsional rigidity of arbitrary triangles. Moreover, what makes the study of extremal domains difficult for the functionals mentioned above, including the P\'olya functional $F$ studied in this paper,  is the competing  symmetries in the problem. While the classical symmetrizations techniques, such as spherical or Steiner  symmetrization, increase the torsional rigidity, they decrease the eigenvalue.  At present there are no general techniques (symmetrization or other types) that give the increasing,  or decreasing, of the product as a single unit. Thus, the results are obtained by developing ad-hoc techniques.  For example, for the lower bound of the product one finds good lower bounds for each quantity involved and similarly for the upper bound of the products. This requires dividing domains into various geometric cases and applying different techniques to different cases.

In the case for triangles our approach is to split the proof into several acute and obtuse cases to prove the lower bound. We rely on various different techniques to obtain the required bounds depending on the cases. We use domain monotonicity to compare with other domains where explicit formulas are known. We also use various inequalities for the first Dirichlet eigenvalue proved  in \cite{Freitas-Siudeja-2010,Siudeja-2007,Laugesen-Siudeja-2011,Siudeja-IU-2010}. We  use the variational characterization of the torsional rigidity (see \eqref{eq:TorRig-variational}) to prove new lower bounds  for the torsional rigidity. In some of the cases, we rely on  Steiner symmetrization to give a bound for the eigenvalue in terms of other triangles. The most difficult cases concerns those of the thin triangles where $F$ approaches the sharp lower bound $\pi^2/24$. In these cases, given in Proposition \ref{prop:Lower-Acute-b-high} and \ref{prop-Obtuse-Lower-Case3}, the new idea is to use a monotonicity result (Lemma \ref{lem:Acute-High-1} and \ref{lem:Techinical-Lower-Mgeq3}) to reduce to a lower bound  for right triangles.

All of the derived bounds are done analytically. Some of the inequalities are shown by proving various technical lemmas on explicit functions. Some of the lemmas are reduced to proving polynomial inequalities, of which we adopt a method of Siudeja given in \cite[Section 5]{Siudeja-IU-2010}.  The upper bound will rely on previous known bounds by Siudeja, Makai and Solynin-Zalgaller. The monotonicity result for the P\'olya functional for rectangles follows from an explicit infinite series expression obtained from the classical  expansion of the Dirichlet heat kernel for rectangles in terms of the eigenvalues and eigenfunctions.

\subsection{Organization of the paper}

The paper is organized as follows.  Section \ref{proof-set-up} gives the geometric description for arbitrary triangles $\triangle_{a,b}$  in terms of the pair of parameters $(a, b)$.  In terms of these parameters, one can describe the cases of (1) \textbf{obtuse}, (2) \textbf{acute}, (3) \textbf{isosceles} and (4) \textbf{right},   triangles. See Figure \ref{figure1}. This section also recalls the exact formulas for the torsion function, torsional rigidity, eigenfunction, eigenvalue and P\'olya functional for the equilateral triangle, one of the few triangles where all the quantities are known.  Section \ref{prelimLB} gives several Lemmas proving lower bounds on quantities that will be used in the various cases for the lower bound estimate in Theorem \ref{thm:MainResult}. Section \ref{sec:Lower-Triangle} proves the  lower bound for Theorem \ref{thm:MainResult} for acute and right triangles.  Section \ref{sec:Lower-Triangle-Obtuse} proves the bound for obtuse triangles. The announced upper bounds are proved in Section \ref{upperbounds}, both for triangles and  tangential quadrilaterals. Section \ref{upperbounds} also contains Proposition \ref{Prop:upper-thinning} which shows the sharpness of the lower bound of Theorem \ref{thm:MainResult} for any sequence of thinning triangles.  Section \ref{CollectAll} collects all the bounds to conclude the proof of Theorem \ref{thm:MainResult} for triangles.  Section \ref{rectanglecase}  proves Theorem  \ref{thm:MainResult} in the case of rectangles and shows the monotonicity as stated in Remark \ref{strongerrectangles}.

\section{\label{sec:Triangle-Set-up}Preliminaries for triangles}

\subsection{Proof set up}\label{proof-set-up}
Consider a triangle $\triangle_{a,b}$ with vertices on $\left(0,0\right)$, $\left(1,0\right)$, $\left(a,b\right)$
with sides of length $1,M=\sqrt{a^{2}+b^{2}}$ and $N=\sqrt{\left(a-1\right)^{2}+b^{2}}$.
By translation, rotation, and scaling invariance of $F\left(D\right)$,
it is enough to consider triangles of the form $\bigtriangleup_{a,b}$
whose admissible set of points $\left(a,b\right)$ come from
\[
\mathcal{T}=\left\{ \left(a,b\right)\in\left[0,\frac{1}{2}\right]\times\left[0,\frac{\sqrt{3}}{2}\right]\mid\left(a-1\right)^{2}+b^{2}\leq1,0\leq a\leq\frac{1}{2}\right\} .
\]
Note that $M\leq N\leq 1$. Let $\gamma$ be the angle between the sides
of length $M,N$, so that by the law of cosines we have $\cos\gamma=\frac{M^{2}+N^{2}-1}{2NM}$. Using this we can observe the following. 
The\textbf{obtuse triangles} correspond to the case when $\frac{\pi}{2}<\gamma<\pi$
which occurs exactly when $\left(a-\frac{1}{2}\right)^{2}+b^{2}<\left(\frac{1}{2}\right)^{2}$
,$0\leq a,b\leq\frac{1}{2}$. The \textbf{right triangles} correspond
to the curve $\left(a-\frac{1}{2}\right)^{2}+b^{2}=\left(\frac{1}{2}\right)^{2}$, $0\leq a,b\leq\frac{1}{2}$.
The \textbf{acute triangles} correspond to $\left(a-1\right)^{2}+b^{2}<1$, $0\leq a\leq\frac{1}{2}$, $0\leq b\leq\frac{\sqrt{3}}{2}$ outside
of the obtuse region. Finally, the \textbf{isoscele triangles} correspond
to those on the part of the circle $\left(a-1\right)^{2}+b^{2}=1$, $0\leq a\leq\frac{1}{2}$, $0\leq b\leq\frac{\sqrt{3}}{2}$ and
also the vertical line $a=\frac{1}{2}$ for $0\leq b\leq\frac{\sqrt{3}}{2}$. See Figure \ref{figure1}.  We will mainly use this characterization when dealing with triangles
that are \textbf{obtuse} so that $\left(a,b\right)\in\mathcal{T}_{\text{obtuse}}$
where 
\[
\mathcal{T}_{\text{obtuse}}=\left\{ \left(a,b\right)\in\left[0,\frac{1}{2}\right]\times\left[0,\frac{1}{2}\right]\mid\left(a-\frac{1}{2}\right)^{2}+b^{2}\leq\frac{1}{4}\right\} .
\]

\begin{figure}[h]
    \centering
    \includegraphics[width=0.6\textwidth]{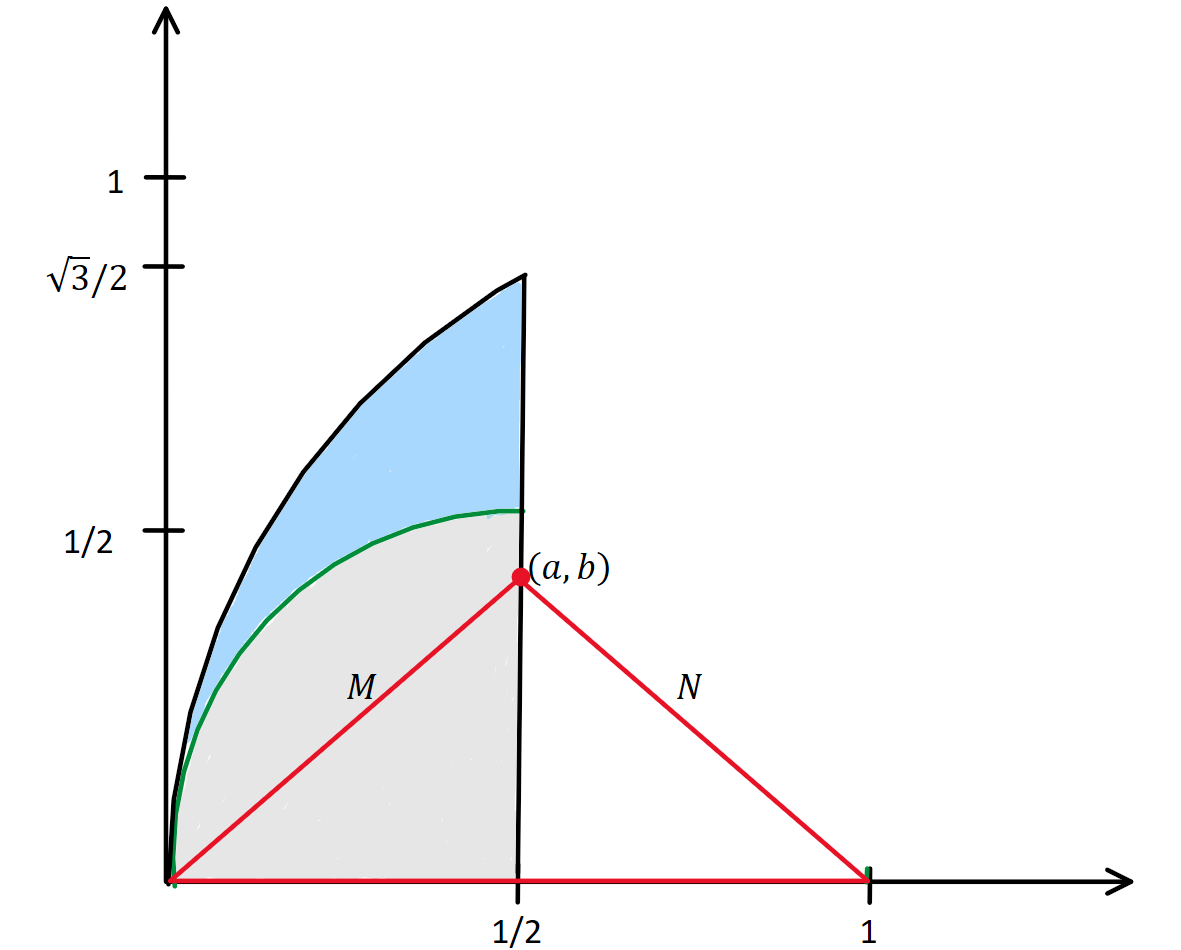}
    \caption{The blue shaded region corresponds to acute triangles.  The grey shaded region corresponds to obtuse triangles. The green inner circular arc corresponds to the right triangles. The black outer circular arc combined with the black vertical line at $a=\frac{1}{2}$ corresponds to the isosceles triangles. An example of an isosceles triangle is given in red. }
 \label{figure1}
\end{figure}


For \textbf{acute and right triangles}, we will use a different
characterization of triangles. In particular, we can write any acute
and right triangle as $\triangle^{M,N}$ with side lengths $1\leq M\leq N$
such that $N\leq\sqrt{M^{2}+1}$. We will still associate this triangle
with the one whose vertices are at $\left(0,0\right)$,$\left(1,0\right)$
and $\left(a,b\right)$ but now take $\left(a,b\right)\in\mathcal{T}_{acute}^{\prime}$ where 
\[
\mathcal{T}_{acute}^{\prime}=\left\{ \left(a,b\right)\mid0\leq a\leq\frac{1}{2},a^{2}+b^{2}\geq1\right\} .
\]


Throughout the paper, and no matter the characterization, we define
$\gamma$ to be the top angle between the sides of length $M,N$.
We define $\beta$ to be the bottom right angle between the sides
of length $N$ and $1$, while $\alpha$ is the bottom left angle
between the sides of length $1$ and $M$. When using the characterization $\mathcal{T}$ it turns out that $1$
is the length of largest side of $\triangle_{a,b}$ and $M\leq N\leq1$. Moreover, using $\mathcal{T}$ we have that $\beta \leq \alpha \leq\gamma$ . When using the characterization $\mathcal{T}^{\prime}$ it turns out
that $1$ is the length of the smallest side of $\triangle_{a,b}$
and $1\leq M\leq N$. Moreover, using $\mathcal{T}^{\prime}$ we have that $ \gamma \leq \beta \leq \alpha$ . When convenient, we denote $\triangle^{M,N}=\triangle_{a,b}$
when we are dealing with the lengths of the sides of the triangle.
Notice that $\left|\triangle_{a,b}\right|=\frac{b}{2}$.

We will use the different characterizations depending on the different
cases that we shall consider and what is most convenient for our computations. The idea
of the proof of Theorem   \ref{thm:MainResult}
is to split the admissible sets $\mathcal{\mathcal{T}},\mathcal{T}^{\prime}$ into various
regions and use different estimating techniques depending on the region.
Often the regions will depend weather $\left(a,b\right)$ is far away
from the equilateral triangle or not.


The equilateral triangle corresponds to $E=\triangle_{\left(\frac{1}{2},\frac{\sqrt{3}}{2}\right)}$
with vertices at $\left(0,0\right),\left(1,0\right),\left(\frac{1}{2},\frac{\sqrt{3}}{2}\right).$
The torsion function for the equilateral triangle $E=\triangle_{\left(\frac{1}{2},\frac{\sqrt{3}}{2}\right)}$
is given by 
\begin{equation}
u_{E}\left(x,y\right)=\frac{1}{2\sqrt{3}}\left(y-\sqrt{3}x\right)\left(y+\sqrt{3}x-\sqrt{3}\right)y,\label{eq:Tor-Equil}
\end{equation}
and a first Dirichlet eigenfunction for $E$ is given by 
\[
\varphi_{E}\left(x,y\right)=\sin\left(\frac{4\pi y}{\sqrt{3}}\right)-\sin\left(2\pi\left(x+\frac{y}{\sqrt{3}}\right)\right)+\sin\left(2\pi\left(x-\frac{y}{\sqrt{3}}\right)\right).
\]

Moreover the following are well known
\[
T\left(E\right)=\frac{\sqrt{3}}{320},\,\,\,\,\,\,\,\,\,\,\,\,\lambda_{1}\left(E\right)=\frac{16\pi^{2}}{3},\,\,\,\,\,\,\,\,\,\,\,\,F\left(E\right)=\frac{T\left(E\right)\lambda_{1}\left(E\right)}{\left|E\right|}=\frac{\pi^{2}}{15}.
\]


\subsection{Preliminary Lower Estimates}\label{prelimLB}

We use various methods to estimate $T\left(D\right)$ . One method
of estimating $T\left(D\right)$ will be through the following variational formula
\begin{equation}
T\left(D\right)=\sup_{v\in H_{0}^{1}(D)\backslash\{0\}}\frac{\left(\int_{D}vdx\right)^{2}}{\int_{D}\left|\nabla v\right|^{2}dx}.\label{eq:TorRig-variational}
\end{equation}
We can find a test function for estimating $T\left(\triangle_{a,b}\right)$
by using a linear transformation of $u_{E}$. In particular, a test function for any triangle is given by  
\begin{align}
v\left(x,y\right) & =u_{E}\left(x-\frac{a-1/2}{b}y,\frac{\sqrt{3}}{2b}y\right)=\frac{3y}{4b^{3}}\left(bx-ay\right)\left(b-bx+(a-1)y\right).\label{eq:Tor-Equil-test-func}
\end{align}
Since the triangle $\bigtriangleup_{a,b}$ is bounded by the lines
$y=\frac{b}{a}x,y=\frac{b}{a-1}\left(x-1\right),y=0$ then it is clear
that $v=0$ on $\partial\bigtriangleup_{a,b}$ and $v\in H_{0}^{1}\left(\bigtriangleup_{a,b}\right)$. This test function is similar to the ones used in \cite{Freitas-2006,Siudeja-2007,Laugesen-Siudeja-2011}  to obtain upper estimates for $\lambda\left(\triangle_{a,b}\right)$ but with $u_E$ replaced by the first eigenfunction $\varphi_E$.

We can then obtain the following estimate on $T\left(\triangle_{a,b}\right)$
with this test function. This bound will help when dealing with triangles
that are closer to the equilateral triangle. 
\begin{lem}
\label{lem:Test-Function-Approach-1}Using the test function $v\left(x,y\right)$
in $(\ref{eq:Tor-Equil-test-func})$ we obtain 
\[
T\left(\bigtriangleup_{a,b}\right)\geq\frac{b^{3}}{80\left(1-a+a^{2}+b^{2}\right)},
\]
for any $\left(a,b\right)\in\mathbb{R}^{2}$.
\end{lem}

\begin{proof}
A computation shows that 
\begin{align*}
T\left(\triangle_{a,b}\right) & \geq\frac{\left(\int_{0}^{b}\int_{ay/b}^{\left(a-1\right)y/b+1}v\left(x,y\right)dxdy\right)^{2}}{\int_{0}^{b}\int_{ay/b}^{\left(a-1\right)y/b+1}\left|\nabla v\right|^{2}dxdy}=\frac{\left(\frac{b}{160}\right)^{2}}{\left(\frac{1-a+a^{2}+b^{2}}{320b}\right)}=\frac{b^{3}}{80\left(1-a+a^{2}+b^{2}\right)}.
\end{align*}
\end{proof}

A circular sector $S\left(\gamma,\rho\right)$ of radius $\rho$ and
angle $\gamma$ turns out to be a good domain to estimate triangles.
The following lower bound on sectors are good for triangles that are
long and thin. The bound will be in terms of $j_{\nu,k}$ which denotes the $k$th positive zero of the Bessel function $J_{\nu }(x)$. We denote $j_{\nu}=j_{\nu,1}$ its first zero. This lemma follows directly from \cite[Theorem 1.3]{Siudeja-IU-2010} and we state it here for easy reference of its explicit bound. 

\begin{lem}
\label{lem:Bartek-Lem1}Let $\left(a,b\right)\in\mathcal{T}_{acute}^{\prime}$ and $\gamma$ be the angle of triangle $\triangle_{a,b}$
between the edges of length $M$ and $N$. Let $S(\gamma,\rho)$ be
the sector with angle $\gamma$ and radius $\rho$ such that $\left|\triangle_{a,b}\right|=\left|S(\gamma,\rho)\right|$.
Then
\[
\lambda_{1}\left(\triangle_{a,b}\right)\geq\frac{\gamma}{b}j_{\pi/\gamma}^{2}.
\]
Moreover, if $\left(a,b\right)\in\mathcal{T}_{obtuse}$ and
$\beta$ is the angle of triangle $\triangle_{a,b}$ between the edges
of length $1$ and $N=\sqrt{\left(1-a\right)^{2}+b^{2}}$ then 
\[
\lambda_{1}\left(\triangle_{a,b}\right)\geq\frac{\beta}{b}j_{\pi/\beta}^{2}.
\]
\end{lem}

\begin{proof}
If $\left(a,b\right)\in\mathcal{T}_{acute}^{\prime}$ then $\gamma$ is the smallest angle. Let $S(\gamma,\rho)$
the sector such that $\left|\triangle_{a,b}\right|=\left|S(\gamma,\rho)\right|$.
Since $\left|\triangle_{a,b}\right|=\frac{b}{2}$ and $\left|S(\gamma,\rho)\right|=\frac{\rho^{2}\gamma}{2}$,
then $\frac{b}{2}=\frac{\rho^{2}\gamma}{2}$ so that $\rho=\sqrt{\frac{b}{\gamma}}$
. Now by \cite[Theorem 1.3]{Siudeja-IU-2010} it is shown that 
\begin{align*}
\lambda_{1}\left(\triangle_{a,b}\right) & \geq\lambda_{1}\left(S(\gamma,\rho)\right)=\frac{j_{\pi/\gamma}^{2}}{\rho^{2}}=\frac{\gamma}{b}j_{\pi/\gamma}^{2},
\end{align*}
where $j_{\nu}$ is the first zero of the Bessel function $J_{\nu}$. If $\left(a,b\right)\in\mathcal{T}_{obtuse}$, then 
the angle $\beta$ is the smallest angle hence the rest of the proof is done similarly.
\end{proof}

Another method we will use throughout the paper will be the domain
monotonicity properties of $T\left(D\right)$ and $\lambda_{1}\left(D\right)$.
It is clear from the variational principal of both $\lambda_{1}\left(D\right)$
and $T\left(D\right)$ that if $D_{1}\subset D_{2}$ then $\lambda_{1}\left(D_{1}\right)\geq\lambda_{1}\left(D_{2}\right)$
while $T\left(D_{1}\right)\leq T\left(D_{2}\right)$. The domain monotonicity
of $T$ also follows easily from the probabilistic definition of $T\left(D\right)=\frac{1}{2}\int_D\mathbb{E}_{x}\left[\tau_{D}\right]dx$. This is clear since if $D_{1}\subset D_{2}$ then a Brownian path
$B_{t}$ started in $D_{1}$ has to exit $D_{1}$ before exiting $D_{2}$. Hence $\mathbb{E}_{x}\left[\tau_{D_{1}}\right]\leq\mathbb{E}_{x}\left[\tau_{D_{2}}\right]$
for $x\in D_{1}$,  which implies $T\left(D_{1}\right)\leq T\left(D_{2}\right)$.

The following bound will be useful when dealing  with  tall and long triangles.  

\begin{lem}[Bound for Acute/Right triangle case]
\label{lem:TorRigbound-1} Consider a triangle $\triangle^{M,N}$
of side lengths $1,M,N$ where $1\leq M\leq N$. Suppose 
(1) $M\geq2$, or (2) $0\leq\gamma\leq\frac{\pi}{4}$
holds.  
Then 

\begin{align*}
T\left(\triangle^{M,N}\right) & \geq\frac{h^{4}}{16}\left(\tan\gamma-\gamma-\frac{124\zeta\left(5\right)\gamma^{4}}{\pi^{5}}\right),
\end{align*}
where $\gamma$ is the angle between the sides of length $M,N$ and $h$
is the altitude of the isosceles triangle $T_{iso}$ of lengths $M,M,c$
with the same angle $\gamma$ between the side lengths $M$ and $M$.
Note that $h$ satisfies $h\geq\sqrt{M^{2}-1/4}$  

\end{lem}

\begin{proof}
Consider the circular sector $$S\left(\alpha,r_{0}\right)=\left\{ \left(r,\phi\right):0<r<r_{0},-\alpha/2<\phi<\alpha/2\right\}.$$ 
It is known that (see \cite[pp. 278-280]{Timoshenko-Goodier} and \cite[Equation (5.6)]{Vandenberg-Ferone-Nitsch-Trombetti-2019a})
\begin{align*}
u_{S\left(\alpha,r_{0}\right)}\left(r,\phi\right) & =\frac{r^{2}}{4}\left(\frac{\cos\left(2\phi\right)}{\cos\alpha}-1\right)\\
 & +\frac{4r_{0}^{2}\alpha^{2}}{\pi^{3}}\sum_{n=1,3,5,\dots}\left(-1\right)^{\left(n+1\right)/2}\frac{\left(\frac{r}{r_{0}}\right)^{n\pi/\alpha}\cos\left(\frac{n\pi\phi}{\alpha}\right)}{n\left(n+\frac{2\alpha}{\pi}\right)\left(n-\frac{2\alpha}{\pi}\right)}
\end{align*}
and 
\begin{align*}
T\left(S\left(\alpha,r_{0}\right)\right) & =\int_{0}^{r_{0}}\int_{-\alpha/2}^{\alpha/2}u_{S\left(\alpha,r_{0}\right)}\left(r,\phi\right)rdrd\phi\\
 & =\frac{r_{0}^{4}}{16}\left(\tan\alpha-\alpha-\frac{128\alpha^{4}}{\pi^{5}}\sum_{n=1,3,5,\dots}\frac{1}{n^{2}\left(n+\frac{2\alpha}{\pi}\right)^{2}\left(n-\frac{2\alpha}{\pi}\right)}\right).
\end{align*}
Recall that $\gamma$ denotes the angle between the sides of length
$M$ and $N$. Consider the isosceles triangle with angle $\gamma$
and side lengths $M$. It is clear that this triangle is inside $\Delta^{M,N}$.
The shortest side of this isosceles triangle cannot have length greater
than $1$.  Thus  its altitude $h$ satisfies $h\geq\sqrt{M^{2}-\frac{1}{4}}$. See Figure \ref{Pic-Sector-1}.

\begin{figure}[h]
    \centering
    \includegraphics[width=0.3\textwidth]{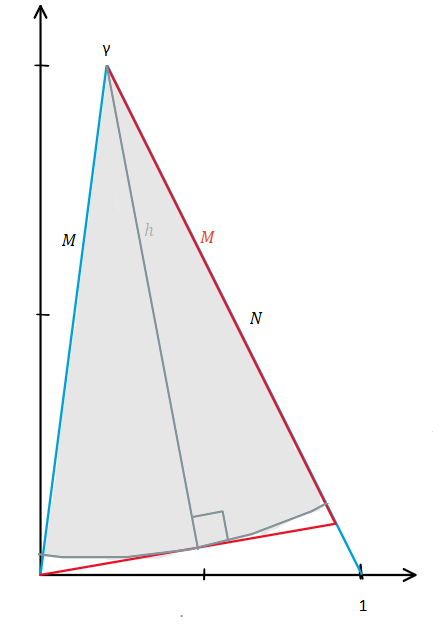}
    \caption{Sector $S(\gamma,h)$}
    \label{Pic-Sector-1}
\end{figure}

Hence the sector $S\left(\gamma,h\right)$ satisfies $\Delta^{M,N}\supset S\left(\gamma,h\right)$
so that 
\begin{align*}
T\left(\Delta^{M,N}\right) & \geq T\left(S\left(\gamma,h\right)\right)\\
 & =\frac{h^{4}}{16}\left(\tan\gamma-\gamma-\frac{128\gamma^{4}}{\pi^{5}}\sum_{n=1,3,5}\frac{1}{n^{2}\left(n+\frac{2\gamma}{\pi}\right)^{2}\left(n-\frac{2\gamma}{\pi}\right)}\right).
\end{align*}

Given $M$, we know that $0\leq\gamma\leq2\sin^{-1}\left(\frac{1}{2M}\right)$
since an isosceles triangles maximizes $\gamma$. 

\underline{Case (1):} 
An elementary computation shows that if $$n\geq\frac{2\sin^{-1}\left(\frac{1}{2M}\right)}{\pi}\left(1+\sqrt{5}\right)$$
then 
\[
\min_{0\leq\gamma\leq2\sin^{-1}\left(\frac{1}{2M}\right)}\left(n+\frac{2\gamma}{\pi}\right)^{2}\left(n-\frac{2\gamma}{\pi}\right)=n^{3}.
\]
Hence, this is true whenever  $n\geq1$ and  $M\geq2$.

\underline{Case (2):} A similar elementary computation shows that 
\[
\min_{0\leq\gamma\leq\frac{\pi}{4}}\left(n+\frac{2\gamma}{\pi}\right)^{2}\left(n-\frac{2\gamma}{\pi}\right)=n^{3}.
\]
as long as   $n\geq\frac{1}{4}\left(1+\sqrt{5}\right)\approx.78$. Hence this minimization problem is true for all $n\geq1$.

In both cases we can use that fact that for all $n\geq1$ and
all admissible $\gamma$,  we have 
\[
n^{2}\left(n+\frac{2\gamma}{\pi}\right)^{2}\left(n-\frac{2\gamma}{\pi}\right)\geq n^{5}
\]
which gives 
\[
T\left(\triangle_{M,N}\right)\geq\frac{h^{4}}{16}\left(\tan\gamma-\gamma-\frac{128\gamma^{4}}{\pi^{5}}\sum_{n=1,3,5}\frac{1}{n^{5}}\right).
\]
Since $$\sum_{n=1,3,5}\frac{1}{n^{5}}=\frac{31\zeta\left(5\right)}{32},$$
we can rewrite 
\begin{align*}
T\left(\triangle_{M,N}\right) & \geq\frac{h^{4}}{16}\left(\tan\gamma-\gamma-\frac{31\zeta\left(5\right)}{32}\frac{128\gamma^{4}}{\pi^{5}}\right)=\frac{h^{4}}{16}\left(\tan\gamma-\gamma-\frac{124\zeta\left(5\right)\gamma^{4}}{\pi^{5}}\right), 
\end{align*}
which is the desired lower bound. 
\end{proof}
We also need the following elementary geometric lemma which is proved here for completeness. 
\begin{lem}
\label{lem:altitude-h}Consider a right triangle $\bigtriangleup_{0,M}=\bigtriangleup^{M,\sqrt{M^{2}+1}}$
and let $\gamma$ be the angle between the sides $M$ and $N=\sqrt{M^{2}+1}$
. Let $h$ be the altitude of the isosceles triangle $T_{iso}$ of
lengths $M,M,c$ with the same angle $\gamma$ between the side lengths
$M$ and $M$. Then 
\[
h=\frac{M}{\sqrt{2}}\sqrt{1+\frac{M}{N}}.
\]
\end{lem}

\begin{proof}
First note that $\bigtriangleup_{0,M}=\bigtriangleup^{M,\sqrt{M^{2}+1}}$.
Recall that $h$ is the altitude between the isosceles triangle $T_{iso}$
of length $M,M,c$ and angle $\gamma$ between the two side lengths
$M$. The law of cosines says that if $\theta$ is the angle between
side lengths $a,b$ , and $c$ is the opposite side of $\theta$ then
$
c^{2}=a^{2}+b^{2}-2ab\cos\theta.
$

First we find the angle $\beta$ of $\bigtriangleup_{0,b}=\bigtriangleup^{M,\sqrt{M^{2}+1}}$
between side lengths $1,N$, and note that $\cos\beta  =\frac{1}{N}$. Let $T_{bottom}$ be the triangle of side length
$c,N-M,1$, (see Figure \ref{Pic-h-calculation})
\begin{figure}[h]
    \centering
    \includegraphics[width=0.4\textwidth]{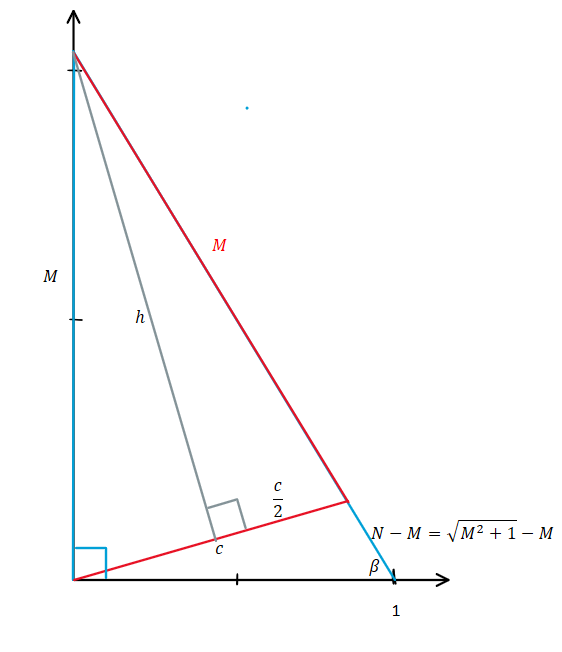}
    \caption{Picture of $T_{iso}$}
    \label{Pic-h-calculation}
\end{figure}
we can solve for $c$; 
\begin{align*}
c^{2} & =1+\left(N-M\right)^{2}-2\left(N-M\right)\cos\beta\\
 & =1+\left(\sqrt{M^{2}+1}-M\right)^{2}-\frac{2\left(\sqrt{M^{2}+1}-M\right)}{\sqrt{M^{2}+1}} =2M^{2}\left(1-\frac{M}{N}\right). 
\end{align*}
Hence,   
$
c=\sqrt{2}M\sqrt{1-\frac{M}{N}}.
$

Let  $T_{right}$ be the right triangle inside $T_{iso}$ of side
lengths $h,M,\frac{\sqrt{2}}{2}M\sqrt{1-\frac{M}{N}}$. Using the
Pythagorean theorem we obtain 
$$
h^{2}+\frac{2}{4}M^{2}\left(1-\frac{M}{N}\right)=M^{2}
$$
so that 
\begin{align*}
h^{2} & =M^{2}-\frac{1}{2}M^{2}\left(1-\frac{M}{N}\right) =\frac{M^{2}}{2}\left(1+\frac{M}{N}\right).
\end{align*}
Hence 
$
h=\frac{M}{\sqrt{2}}\sqrt{1+\frac{M}{N}},
$ as desired. 
\end{proof}


\section{Proof of Theorem \ref{thm:MainResult}:
Lower Bound for Acute and Right Triangles}\label{sec:Lower-Triangle}

We will split the proof into two main cases. See Figure \ref{Acute-LowerBound-v2} for a picture of the regions for $(a,b)$.

\begin{figure}[h]
    \centering
    \includegraphics[width=0.4\textwidth]{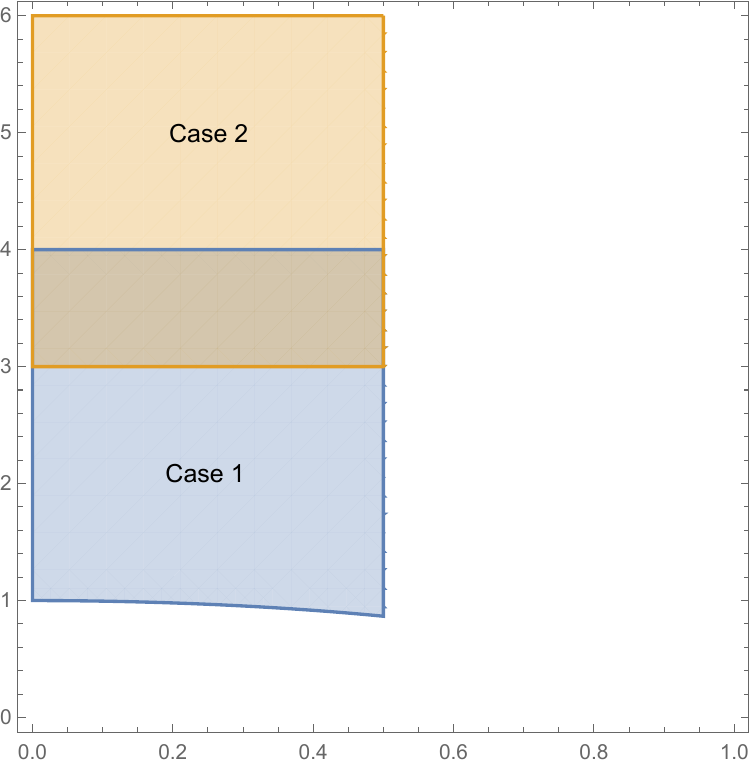}
    \caption{Cases for Lower bound for acute and right triangles }
    \label{Acute-LowerBound-v2}
\end{figure}

We first consider acute and right triangles that are close to the
isosceles right triangle and equilateral triangle. 

\begin{prop}[Case 1]
\label{prop:Lower-Acute-b-low}If $\left(a,b\right)\in\mathcal{T}_{acute}^{\prime}$ then 
\[
\frac{\pi^{2}}{24}< F\left(\bigtriangleup_{a,b}\right),\text{ for }\frac{\sqrt{3}}{2}\leq b\leq4.
\]
\end{prop}

\begin{proof}

We split the rest of the proof into two cases with some overlap, of which there are some overlap. 

\textbf{\underline{Case 1a:}} Consider the region $\frac{\sqrt{3}}{2}\leq b\leq2.9$. Recall that this includes the equilateral triangle $E=\triangle_{\frac{1}{2},\frac{\sqrt{3}}{2}}$. 
By a result of Freitas and Siudeja in \cite[Corollary 4.1]{Freitas-Siudeja-2010} we have the following bound for
the eigenvalue of a triangle, 
\[
\lambda_{1}\left(\triangle_{a,b}\right)\geq\pi^{2}\left(\frac{1}{d}+\frac{1}{h}\right)^{2},
\]
where $d$ is the diameter and $h$ is the height perpendicular to
its longest side. If $\left(a,b\right)\in\mathcal{T}_{acute}^{\prime}$ then $d=N=\sqrt{\left(a-1\right)^{2}+b^{2}}$.  Since $\frac{b}{2}=\left|\triangle_{a,b}\right|=\frac{1}{2}Nh$, then
$h=\frac{b}{N}$ so that 
\begin{align*}
\lambda_{1}\left(\triangle_{a,b}\right) & \geq\pi^{2}\left(\frac{1}{\sqrt{\left(a-1\right)^{2}+b^{2}}}+\frac{\sqrt{\left(a-1\right)^{2}+b^{2}}}{b}\right)^{2}\\
 & =\pi^{2}\frac{\left(\left(a-1\right)^{2}+b^{2}+b\right)^{2}}{b^{2}\left(\left(a-1\right)^{2}+b^{2}\right)}.
\end{align*}

Putting this together with the bound $$T\left(\triangle_{a,b}\right)\geq\frac{b^{3}}{80\left[1-a+a^{2}+b^{2}\right]}$$ from Lemma \ref{lem:Test-Function-Approach-1}, gives 
\begin{align*}
F\left(\triangle_{a,b}\right)=\frac{\lambda_{1}\left(\triangle_{a,b}\right)T\left(\triangle_{a,b}\right)}{\left|\triangle_{a,b}\right|} & \geq\frac{\pi^{2}\frac{\left(\left(a-1\right)^{2}+b^{2}+b\right)^{2}}{b^{2}\left(\left(a-1\right)^{2}+b^{2}\right)}\cdot\frac{b^{3}}{80\left[1-a+a^{2}+b^{2}\right]}}{\frac{b}{2}}\\
 & =\frac{\pi^{2}}{24}\frac{3}{5}\frac{\left(\left(a-1\right)^{2}+b^{2}+b\right)^{2}}{\left(\left(a-1\right)^{2}+b^{2}\right)\left(\left(a-1\right)^{2}+b^{2}+a\right)}.
\end{align*}
Define $$g\left(a,b\right)=\frac{3}{5}\frac{\left(\left(a-1\right)^{2}+b^{2}+b\right)^{2}}{\left(\left(a-1\right)^{2}+b^{2}\right)\left(\left(a-1\right)^{2}+b^{2}+a\right)}.$$
An elementary calculation shows that 
\[
g\left(a,b\right)\geq g\left(\frac{1}{2},2.9\right)=\frac{501126}{495785}>1,\text{ for }0\leq a\leq\frac{1}{2},\frac{\sqrt{3}}{2}\leq b\leq2.9,
\]
which gives $F\left(\triangle_{a,b}\right)\geq\frac{\pi^{2}}{24}$
as needed.

\textbf{\underline{Case 1b:}} We consider the region $1\leq b\leq4$. 

Here we estimate $\lambda_{1}\left(\triangle_{\frac{1}{2},b}\right)$
differently. First, using Steiner symmetrization with respect to the horizontal $x-$axis
we have that since $\left|\triangle_{a,b}\right|=\left|\triangle_{\frac{1}{2},b}\right|,$ then
$$
\lambda_{1}\left(\triangle_{a,b}\right)\geq\lambda_{1}\left(\triangle_{\frac{1}{2},b}\right).
$$
Let $\gamma_{b}$ be the smallest angle between the sides
$M_{\frac{1}{2},b},N_{\frac{1}{2},b}$ of $\triangle_{\frac{1}{2},b}$.
Let $S\left(\gamma_{b},\rho\right)$ be the circular sector such that
$\left|\triangle_{a,b}\right|=\left|S\left(\gamma,\rho\right)\right|$.
By Lemma \ref{lem:Bartek-Lem1} we have that
\begin{align*}
\lambda_{1}\left(\triangle_{a,b}\right) & \geq\frac{\gamma_{b}}{b}j_{\pi/\gamma_{b}}^{2}.
\end{align*}

By Lemma \ref{lem:Test-Function-Approach-1} and using the fact that 
 $0< a\leq\frac{1}{2}$ so that  $-a+a^{2}=a\left(a-1\right)<0$,
we have
\[
T\left(\triangle_{a,b}\right)\geq\frac{b^{3}}{80\left[1-a+a^{2}+b^{2}\right]}> \frac{b^{3}}{80\left(1+b^{2}\right)}.
\]

Putting these bounds together we obtain, 
\begin{align*}
F\left(\triangle_{a,b}\right)=  \frac{T\left(\triangle_{a,b}\right)\lambda_{1}\left(\triangle_{a,b}\right)}{\left|\triangle_{a,b}\right|} & \geq\frac{\frac{b^{3}}{80\left(1+b^{2}\right)}\cdot\frac{\gamma_{b}}{b}j_{\pi/\gamma_{b}}^{2}}{\frac{b}{2}}\\
 & =\frac{b}{40\left(1+b^{2}\right)}\cdot\gamma_{b}j_{\pi/\gamma_{b}}^{2}.
\end{align*}
The  zeros of Bessel function can be bound by $j_{\nu,k}>\nu-\frac{a_{k}}{2^{1/3}}\nu^{1/3}$
given in \cite{Qu-Wong-1999} where $a_k$ is the $k$th negative zero of the Airy function $\text{Ai}(x)$. We then have that
\[
j_{\pi/\gamma_{b}}^{2}\geq\left(\frac{\pi}{\gamma_{b}}-\frac{a_{1}}{2^{1/3}}\left(\frac{\pi}{\gamma_{b}}\right)^{1/3}\right)^{2}.
\]

Using the known fact that  $-a_{1}\geq2.338107>2.3=:k$, it follows that 
\begin{align*}
F\left(\triangle_{a,b}\right) & \geq\frac{b\gamma_{b}}{40\left(1+b^{2}\right)}\left(\frac{\pi}{\gamma_{b}}+\frac{k}{2^{1/3}}\left(\frac{\pi}{\gamma_{b}}\right)^{1/3}\right)^{2}.
\end{align*}
A simple computation leads to 
$$\gamma_{b}=2\sin^{-1}\left(\frac{1}{2\sqrt{b^{2}+\frac{1}{4}}}\right)=2\tan^{-1}\left(\frac{1}{2b}\right).$$
Making the substitution by letting $x=\tan^{-1}\left(\frac{1}{2b}\right)$
leads to  $b=\frac{1}{2\tan x}$and $\gamma_{b}=2x$.  Hence 
\begin{align*}
F\left(\triangle_{a,b}\right) & \geq\frac{\frac{2x}{2\tan x}}{40\left(1+\frac{1}{4\left(\tan x\right)^{2}}\right)}\left(\frac{\pi}{2x}+\frac{k}{2^{1/3}}\left(\frac{\pi}{2x}\right)^{1/3}\right)^{2}\\
 & =\frac{x\tan x}{40\left(\frac{1}{4}+\left(\tan x\right)^{2}\right)}\left(\frac{\pi}{2x}+\frac{k}{2^{1/3}}\left(\frac{\pi}{2x}\right)^{1/3}\right)^{2}.
\end{align*}
Now note that
\begin{align*}
1\leq b\leq4 & \iff1\leq\frac{1}{2\tan x}\leq4\\
 & \iff2\leq\frac{1}{\tan x}\leq8\\
 & \iff\tan^{-1}\left(1/8\right)\leq x\leq\tan^{-1}\left(1/2\right).
\end{align*}
The following bounds can be obtained using a repeated application of $$\tan(x)=\int_{0}^{x}\left(1+\tan^{2}(t)dt\right)$$
and the fact that $\frac{d\tan x}{dx}=1+\tan^{2}(x)$:

\begin{align*}
\tan x & \geq x+\frac{x^{3}}{3}+\frac{2x^{5}}{15},0<x<\frac{\pi}{2} ,\\
\tan x & \leq x+\frac{x^{3}}{3}+\frac{2x^{5}}{5},0<x<1. 
\end{align*}
Hence
\begin{align*}
F\left(\triangle_{a,b}\right) & \geq\frac{x\left(x+\frac{x^{3}}{3}+\frac{2x^{5}}{15}\right)}{40\left(\frac{1}{4}+\left(x+\frac{x^{3}}{3}+\frac{2x^{5}}{5}\right)^{2}\right)}\left(\frac{\pi}{2x}+\frac{k}{2^{1/3}}\left(\frac{\pi}{2x}\right)^{1/3}\right)^{2}\\
 & =\frac{\pi^{2}}{24}\cdot\frac{3}{5}\frac{x\left(x+\frac{x^{3}}{3}+\frac{2x^{5}}{15}\right)}{\left(1+4\left(x+\frac{x^{3}}{3}+\frac{2x^{5}}{5}\right)^{2}\right)}\left(\frac{1}{x}+\frac{k2^{1/3}}{\pi^{2/3}}\frac{1}{x^{1/3}}\right)^{2}=: \frac{\pi^2}{24}f(x).
\end{align*}
This shows that for any $0\leq a\leq\frac{1}{2}$ and $1\leq b\leq4$
we have 
$
F\left(\triangle_{a,b}\right)>\frac{\pi^{2}}{24}\cdot f\left(x\right)
$.
To prove $F\left(\triangle_{a,b}\right)>\frac{\pi^{2}}{24}$ for
this range, it suffices to prove that the function $f(x)$ satisfies 
\[
f\left(x\right)\geq1,\text{ for }\tan^{-1}\left(\frac{1}{8}\right)\leq x\leq\tan^{-1}\left(1/2\right),
\]
which is done in the following Lemma. 
\end{proof}
\begin{lem}
\label{lem:technical-func-ineq-1}The function 
\[
f(x)=\frac{3}{5}\frac{x\left(x+\frac{x^{3}}{3}+\frac{2x^{5}}{15}\right)}{\left(1+4\left(x+\frac{x^{3}}{3}+\frac{2x^{5}}{5}\right)^{2}\right)}\left(\frac{1}{x}+\frac{k2^{1/3}}{\pi^{2/3}}\frac{1}{x^{1/3}}\right)^{2},
\]
satisfies $f(x)\geq1$,  for $\tan^{-1}\left(\frac{1}{8}\right)\leq x\leq\tan^{-1}\left(1/2\right)$. 
\end{lem}

\begin{proof}
We prove the inequality for $x\in\left(\tan^{-1}\left(\frac{1}{8}\right),\tan^{-1}\left(1/2\right)\right)\subset\left(0.12, 0.464\right)$.
Note that
\begin{align*}
f(x) & =\frac{3}{5}\frac{\left(x+\frac{x^{3}}{3}+\frac{2x^{5}}{15}\right)}{x\left(1+4\left(x+\frac{x^{3}}{3}+\frac{2x^{5}}{5}\right)^{2}\right)}\left(1+\frac{k2^{1/3}}{\pi^{2/3}}x^{\frac{2}{3}}\right)^{2}
\end{align*}
then use $x\mapsto x^{3}$ so that it suffices to show
\[
\frac{3}{5}\frac{\left(x^{3}+\frac{x^{9}}{3}+\frac{2x^{15}}{15}\right)}{x^{3}\left(1+4\left(x^{3}+\frac{x^{9}}{3}+\frac{2x^{15}}{5}\right)^{2}\right)}\left(1+\frac{k2^{1/3}}{\pi^{2/3}}x^{2}\right)^{2}\geq1, 
\]
for $x\in\left(\left(0.12\right)^{1/3},\left(0.464\right)^{1/3}\right)\subset\left(0.49, 0.775\right)$.
Rewriting this as a polynomial inequality, it suffices to show that 
\[
0\geq5x^{3}\left(1+4\left(x^{3}+\frac{x^{9}}{3}+\frac{2x^{15}}{5}\right)^{2}\right)-3\left(x^{3}+\frac{x^{9}}{3}+\frac{2x^{15}}{15}\right)\left(1+\frac{k2^{1/3}}{\pi^{2/3}}x^{2}\right)^{2}=:P_1\left(x\right). 
\]
Expanding with $k=\frac{23}{10}$ we have 
\begin{align*}
P_{1}(x) & =2x^{3}-\frac{69\cdot2^{1/3}}{5\pi^{2/3}}x^{5}-\frac{1587}{50\cdot2^{1/3}\pi^{4/3}}x^{7}+19x^{9}\\
 & -\frac{23\cdot2^{1/3}}{5\pi^{2/3}}x^{11}-\frac{529}{50\cdot2^{1/3}\pi^{4/3}}x^{13}+\frac{194}{15}x^{15}\\
 & -\frac{46\cdot2^{1/3}}{25\pi^{2/3}}x^{17}-\frac{529}{125\cdot2^{1/3}\pi^{4/3}}x^{19}\\
 & +\frac{164}{9}x^{21}+\frac{16}{3}x^{27}+\frac{16}{5}x^{33}. 
\end{align*}
Then using $x\mapsto\left(x+0.49\right)$, it suffices to show that the polynomial

\begin{equation}\label{polyineq-1}
P_{2}(x)=P_{1}(x+.49)
\end{equation}
satisfies $P_{2}\left(x\right)\leq0$, for $x\in\left(0, 0.285\right)$. We now use the Siudeja algorithm described in Section \ref{sec:Polynomial-Inequality} to show
$P_{2}\left(x\right)\leq0$,  as desired. This algorithm was introduced by Siudeja in \cite{Siudeja-IU-2010} and it allows us to show any polynomial is negative on an interval $(0,a)$ given that the interval is small enough. Using the algorithm in Section \ref{Sec:Lemma1} shows the desired inequality. 
\end{proof}
In the following, we consider acute and right triangles that are long and thin. These triangles are far from the equilateral triangle and approach the degenerating lower bound of $\frac{\pi^{2}}{24}$. We bound the torsional rigidity and principal eigenvalue using sectors with bounds given in Lemma \ref{lem:TorRigbound-1} and Lemma \ref{lem:Bartek-Lem1}. Afterwards, one of the key ideas will be to use monotonicity results given in Lemmas \ref{lem:Acute-High-1} and \ref{lem:Techinical-Lower-Mgeq3} to reduce to a lower bound for right triangles.

\begin{prop}[Case 2]
\label{prop:Lower-Acute-b-high}If $\left(a,b\right)\in\mathcal{T}_{acute}^{\prime}$ then
\[
\frac{\pi^{2}}{24}< F\left(\bigtriangleup_{a,b}\right),\text{ for }b\geq3.
\]
\end{prop}

\begin{proof}
To estimate $T\left(\triangle_{a,b}\right)=T\left(\triangle^{M,N}\right)$
we will use Lemma \ref{lem:TorRigbound-1}. As in Lemma \ref{lem:TorRigbound-1}, let $h$ be the altitude between the isosceles triangle $T_{iso}$
of length $M,M,1$ and angle $\gamma=\gamma\left(\triangle_{a,b}\right)$
between the two side lengths $M$. 

Since $M=\sqrt{a^{2}+b^{2}}\geq3$ , by Lemma \ref{lem:TorRigbound-1} we have that
\[
T\left(\triangle_{a,b}\right)\geq\frac{h^{4}}{16}\left(\tan\gamma-\gamma-\frac{124\zeta\left(5\right)\gamma^{4}}{\pi^{5}}\right),
\]
where $\gamma=\gamma\left(\triangle_{a,b}\right)$. Recall that by
Lemma \ref{lem:Bartek-Lem1} we have $\lambda_{1}\left(\triangle_{a,b}\right)\geq\frac{\gamma}{b}j_{\pi/\gamma}^{2}$
and that 
$$j_{\pi/\gamma}^{2}>\left(\frac{\pi}{\gamma}+\frac{k}{2^{1/3}}\left(\frac{\pi}{\gamma}\right)^{1/3}\right)^{2},$$ where
$k=2.338107$. Thus  
\[
\lambda_{1}\left(\triangle_{a,b}\right)\geq\frac{\gamma}{b}\left(\frac{\pi}{\gamma}+\frac{k}{2^{1/3}}\left(\frac{\pi}{\gamma}\right)^{1/3}\right)^{2}.
\]
Putting these bounds together we have that 
\begin{align}
  F\left(\triangle_{a,b}\right)& >\frac{h^{4}}{16}\left(\tan\gamma-\gamma-\frac{124\zeta\left(5\right)\gamma^{4}}{\pi^{5}}\right)\frac{\gamma}{b}\left(\frac{\pi}{\gamma}+\frac{k}{2^{1/3}}\left(\frac{\pi}{\gamma}\right)^{1/3}\right)^{2}\frac{1}{b/2}.\label{prop:Lower-Acute-b-high-0}
\end{align}
Note that given a fixed $b$, we have that for any $0\leq a \leq 1/2$,  
\[
h\left(\text{right triangle}\right)=h\left(\triangle_{0,b}\right)\leq h\left(\triangle_{a,b}\right)\leq h\left(\triangle_{\frac{1}{2},b}\right)=h\left(\text{isosceles}\text{ triangle}\right). 
\]
Hence by Lemma \ref{lem:altitude-h}
\begin{equation}
h\left(\triangle_{a,b}\right)\geq h\left(\triangle_{0,b}\right)=\frac{b}{\sqrt{2}}\sqrt{1+\frac{b}{\sqrt{b^{2}+1}}}.\label{prop:Lower-Acute-b-high-1}
\end{equation}
Using $(\ref{prop:Lower-Acute-b-high-1})$ in $(\ref{prop:Lower-Acute-b-high-0})$
we have that 
\begin{align}
 & F\left(\triangle_{a,b}\right)\nonumber \\
 & \geq\frac{\left(\frac{b}{\sqrt{2}}\sqrt{1+\frac{b}{\sqrt{b^{2}+1}}}\right)^{4}}{16}\left(\tan\gamma-\gamma-\frac{124\zeta\left(5\right)\gamma^{4}}{\pi^{5}}\right)\frac{\gamma}{b}\left(\frac{\pi}{\gamma}+\frac{k}{2^{1/3}}\left(\frac{\pi}{\gamma}\right)^{1/3}\right)^{2}\frac{1}{b/2}\label{prop:Lower-Acute-b-high-1-2}
\end{align}

Note that for a given $a$ and fixed $b$, we have for any $0\leq a \leq 1/2$, 
\[
\gamma\left(\text{right triangle}\right)=\gamma\left(\triangle_{0,b}\right)\leq\gamma\left(\triangle_{a,b}\right)\leq\gamma\left(\triangle_{\frac{1}{2},b}\right)=\gamma\left(\text{isosceles triangle}\right).
\]
This gives  
\begin{equation}
\gamma\left(\triangle_{a,b}\right)\geq\gamma\left(\triangle_{0,b}\right)=\tan^{-1}\left(\frac{1}{b}\right).\label{prop:Lower-Acute-b-high-2}
\end{equation}
Also recall that $\gamma\left(\triangle_{\frac{1}{2},b}\right)=2\tan^{-1}\left(\frac{1}{2b}\right)$
so that $$\gamma\in\left(\tan^{-1}\left(\frac{1}{b}\right),2\tan^{-1}\left(\frac{1}{2b}\right)\right).$$
Since $\tan^{-1}\left(\frac{1}{b}\right),2\tan^{-1}\left(\frac{1}{2b}\right)$
are decreasing in $b$, for all of $b \geq 3$,  we have $$\gamma\in\left(0,2\tan^{-1}\left(\frac{1}{6}\right)\right)\subset\left(0, 0.34\right).$$ Using $(\ref{prop:Lower-Acute-b-high-2})$ in $(\ref{prop:Lower-Acute-b-high-1-2})$
and the fact that the function
\[
\gamma\mapsto\left(\tan\gamma-\gamma-\frac{124\zeta\left(5\right)\gamma^{4}}{\pi^{5}}\right)\gamma\left(\frac{\pi}{\gamma}+\frac{k}{2^{1/3}}\left(\frac{\pi}{\gamma}\right)^{1/3}\right)^{2}
\]
is increasing for $\gamma\in\left(0,0.7\right),$ Lemma \ref{lem:Acute-High-1} gives that 
\begin{align*}
 & F\left(\triangle_{a,b}\right)\\
 & >\frac{\pi^{2}}{24}\cdot\frac{3}{4}b^{2}\left(1+\frac{b}{\sqrt{b^{2}+1}}\right)^{2}\frac{1}{\gamma_{b}}\left(1+\frac{k}{2^{1/3}}\frac{\gamma_{b}^{2/3}}{\pi^{2/3}}\right)^{2}\left(\tan\gamma_{b}-\gamma_{b}-\frac{124\zeta\left(5\right)\gamma_{b}^{4}}{\pi^{5}}\right),
\end{align*}
where $\gamma_{b}:=\tan^{-1}\left(\frac{1}{b}\right)$ 

By Lemma \ref{lem:Techinical-Lower-Mgeq3} we know the function 
\[
f\left(b\right)=\frac{3}{4}b^{2}\left(1+\frac{b}{\sqrt{b^{2}+1}}\right)^{2}\frac{1}{\gamma_{b}}\left(1+\frac{k}{2^{1/3}}\frac{\gamma_{b}^{2/3}}{\pi^{2/3}}\right)^{2}\left(\tan\gamma_{b}-\gamma_{b}-\frac{124\zeta\left(5\right)\gamma_{b}^{4}}{\pi^{5}}\right)
\]
satisfies $f\left(b\right)\geq1$,  for $b\geq3.$ This gives  the desired
result. 
\end{proof}
\begin{lem}
\label{lem:Acute-High-1}The function 
\[
f\left(x\right)=\left(\tan x-x-\frac{124\zeta\left(5\right)x^{4}}{\pi^{5}}\right)x\left(\frac{\pi}{x}+\frac{c_{1}}{2^{1/3}}\left(\frac{\pi}{x}\right)^{1/3}\right)^{2},
\]
where $c_{1}=2.338107$, is increasing in the interval $\left(0,0.7\right)$. 
\end{lem}

\begin{proof}
Making the substitution $x\mapsto x^{3}$,
it suffices to show 
\begin{align*}
f_{1}\left(x\right) & =\left(\tan x^{3}-x^{3}-\frac{124\zeta\left(5\right)x^{12}}{\pi^{5}}\right)x^{3}\left(\frac{\pi}{x^{3}}+\frac{c_{1}}{2^{1/3}}\frac{\pi^{1/3}}{x}\right)^{2}\\
 & =\left(\tan x^{3}-x^{3}-\frac{124\zeta\left(5\right)x^{12}}{\pi^{5}}\right)\frac{1}{x^{3}}\left(\pi+\frac{c_{1}\pi^{1/3}}{2^{1/3}}x^{2}\right)^{2}
\end{align*}
is increasing for $x\in\left(0,\left(0.7\right)^{1/3}\right)\approx\left(0, 0.8879\right)$. 

We know that 
\[
\tan x=x+\frac{x^{3}}{3}+\frac{2}{15}x^{5}+\frac{17}{315}x^{7}+R_{1}\left(x\right),
\]
where the remainder term $R_{1}\left(x\right)=\sum_{i=9}^{\infty}a_{i}x^{i}$
satisfies $a_{i}=0$, when $i$ is even and $a_{i}>0$, when $i$ is
odd and it converges on $\left|x\right|<\frac{\pi}{2}=1.5$. Hence
\[
\tan x^{3}=x^{3}+\frac{x^{9}}{3}+\frac{2}{15}x^{15}+\frac{17}{315}x^{21}+R_{2}\left(x\right),
\]
where the remainder term $R_{2}\left(x\right)=\sum_{i=27}^{\infty}a_{i}x^{i}$
$a_{i}\geq0$ when $i$ is odd which converges on $\left|x\right|<\left(\frac{\pi}{2}\right)^{1/3}\approx1.16$. 
Hence 
\begin{align*}
f_{1}\left(x\right) & =\left(\frac{x^{9}}{3}+\frac{2}{15}x^{15}+\frac{17}{315}x^{21}+R_{2}\left(x\right)-\frac{124\zeta\left(5\right)x^{12}}{\pi^{5}}\right)\frac{1}{x^{3}}\left(\pi+\frac{c_{1}\pi^{1/3}}{2^{1/3}}x^{2}\right)^{2}\\
 & =\left(\frac{x^{6}}{3}+\frac{2}{15}x^{12}+\frac{17}{315}x^{18}+R_{3}\left(x\right)-\frac{124\zeta\left(5\right)x^{9}}{\pi^{5}}\right)\left(\pi+\frac{c_{1}\pi^{1/3}}{2^{1/3}}x^{2}\right)^{2}\\
 & =\left(\frac{x^{6}}{3}+c_{2}x^{9}+\frac{2}{15}x^{12}+\frac{17}{315}x^{18}+R_{3}\left(x\right)\right)\left(\pi+\frac{c_{1}\pi^{1/3}}{2^{1/3}}x^{2}\right)^{2}, 
\end{align*}
where $R_{3}(x)=\sum_{i=24}^{\infty}a_{i}x^{i}$ and $c_{2}=-\frac{124\zeta\left(5\right)}{\pi^{5}}<0$.
Expanding out we have 
\[
f_{1}(x)=P_{1}(x)+P_{2}(x)
\]
where $$P_{2}(x)=R_{3}\left(x\right)\left(\pi^{2}+2^{2/3}c_{1}\pi^{4/3}x^{2}+c_{1}^{2}\left(\frac{\pi}{2}\right)^{2/3}x^{4}\right)$$ and
\begin{align}
\nonumber P_{1}(x) & =\frac{\pi^{2}}{3}x^{6}+\frac{2^{1/3}c_{1}}{3}\pi^{4/3}x^{8}+c_{2}\pi^{2}x^{9}\\
\nonumber & +\frac{c_{1}^{2}}{3}\left(\frac{\pi}{2}\right)^{2/3}x^{10}+2^{2/3}c_{1}c_{2}\pi^{4/3}x^{11}+\frac{2\pi^{2}}{15}x^{12}\\
\nonumber & +c_{1}^{2}c_{2}\left(\frac{\pi}{2}\right)^{2/3}x^{13}+\frac{2}{15}2^{1/3}c_{1}\pi^{4/3}x^{14}+\frac{2^{1/3}c_{1}^{2}\pi^{2/3}}{15}x^{16}\\
\nonumber & +\frac{17\pi^{2}}{315}x^{18}+\frac{17}{315}2^{2/3}c_{1}\pi^{4/3}x^{20}\\
 & +\frac{17}{315}c_{1}^{2}\left(\frac{\pi}{2}\right)^{2/3}x^{22}.\label{polyineq-2}
\end{align}
The polynomial $P_{1}(x)$ is certainly increasing for $x\in\left(0, 0.888\right),$ again applying 
Siudeja's  algorithm (see Section \ref{Sec:Lemma2})  to  $P_{1}^{\prime}(x)$ to show that $-P_{1}^{\prime}(x)\leq0$. Moreover, the polynomial $$P_{2}(x)=R_{3}\left(x\right)\left(\pi^{2}+2^{2/3}c_{1}\pi^{4/3}x^{2}+c_{1}^{2}\left(\frac{\pi}{2}\right)^{2/3}x^{4}\right)$$
has positive powers of $x$ with positive coefficients which means
$P_{2}(x)$ is increasing. This shows $f_{1}(x)$ is increasing on the desired interval as needed.  
\end{proof}
\begin{lem}
\label{lem:Techinical-Lower-Mgeq3} The function
\[
f\left(b\right)=\frac{3}{4}\frac{b^{2}}{\gamma_{b}}\left(1+\frac{b}{\sqrt{b^{2}+1}}\right)^{2}\left(1+\frac{c_{1}}{2^{1/3}}\frac{\gamma_{b}^{2/3}}{\pi^{2/3}}\right)^{2}\left(\frac{1}{b}-\gamma_{b}-\frac{124\zeta\left(5\right)\gamma_{b}^{4}}{\pi^{5}}\right),
\]
where $\gamma_{b}=\tan^{-1}\left(\frac{1}{b}\right)$ and $c_{1}=2.338107$
satisfies $f\left(b\right)\geq1$ for all $b\geq3$. 
\end{lem}

\begin{proof}
We make the substitution $x=\tan^{-1}\left(\frac{1}{b}\right)$ so
that $b=\frac{1}{\tan x}$ hence we define the function $h$ for $0\leq x\leq\tan^{-1}\left(1/3\right)\approx 0.32$ by 
\begin{align*}
h\left(x\right) & =\frac{3}{4}\frac{1}{x\left(\tan x\right)^{2}}\left(1+\frac{1}{\tan x\sqrt{\frac{1}{\left(\tan x\right)^{2}}+1}}\right)^{2}\left(1+\frac{c_{1}}{2^{1/3}\pi^{2/3}}x^{2/3}\right)^{2}\left(\tan x-x-\frac{124\zeta\left(5\right)}{\pi^{5}}x^{4}\right)\\
 & =\frac{3}{4}\frac{1}{x\left(\tan x\right)^{2}}\left(1+\frac{1}{\sqrt{1+\left(\tan x\right)^{2}}}\right)^{2}\left(1+\frac{c_{1}}{2^{1/3}\pi^{2/3}}x^{2/3}\right)^{2}\left(\tan x-x-\frac{124\zeta\left(5\right)}{\pi^{5}}x^{4}\right)\\
 & =\frac{3}{4}\frac{1}{x\left(\tan x\right)^{2}}\left(1+\frac{1}{\left|\sec\theta\right|}\right)^{2}\left(1+\frac{c_{1}}{2^{1/3}\pi^{2/3}}x^{2/3}\right)^{2}\left(\tan x-x-\frac{124\zeta\left(5\right)}{\pi^{5}}x^{4}\right)\\
 & =\frac{3}{4}\frac{\left(1+\cos(x)\right)^{2}}{x\left(\tan x\right)^{2}}\left(1+\frac{c_{1}}{2^{1/3}\pi^{2/3}}x^{2/3}\right)^{2}\left(\tan x-x-\frac{124\zeta\left(5\right)}{\pi^{5}}x^{4}\right).
\end{align*}
Now since $$\left(1+\cos(x)\right)^{2}=4\cos\left(\frac{x}{2}\right)^{4}$$
we have
\[
h\left(x\right)=3\frac{\cos\left(\frac{x}{2}\right)^{4}}{x\left(\tan x\right)^{2}}\left(1+\frac{c_{1}}{2^{1/3}\pi^{2/3}}x^{2/3}\right)^{2}\left(\tan x-x-\frac{124\zeta\left(5\right)}{\pi^{5}}x^{4}\right).
\]

Using the elementary bounds  $1-\frac{x^{2}}{2}\leq\cos(x)$,  for $x\leq2$ so that $1-\frac{x^{2}}{8}\leq\cos\left(\frac{x}{2}\right)$ for $x\leq4$ 
and 
\begin{align*}
\tan x & \geq x+\frac{x^{3}}{3},0<x<\frac{\pi}{2},\\
\tan x & \leq x+\frac{x^{3}}{3}+\frac{2x^{5}}{5},0<x<1, 
\end{align*}
we have that $$\frac{1}{\left(\tan x\right)^{2}}\geq\frac{1}{\left(x+\frac{x^{3}}{3}+\frac{2x^{5}}{5}\right)^{2}}.$$
Thus, 
\begin{align*}
h\left(x\right) & \geq3\frac{\left(1-\frac{x^{2}}{8}\right)^{4}}{x\left(x+\frac{x^{3}}{3}+\frac{2x^{5}}{5}\right)^{2}}\left(1+\frac{c_{1}}{2^{1/3}\pi^{2/3}}x^{2/3}\right)^{2}\left(x+\frac{x^{3}}{3}-x-\frac{124\zeta\left(5\right)}{\pi^{5}}x^{4}\right)\\
 & =3\frac{\left(1-\frac{x^{2}}{8}\right)^{4}}{\left(x+\frac{x^{3}}{3}+\frac{2x^{5}}{5}\right)^{2}}\left(1+\frac{c_{1}}{2^{1/3}\pi^{2/3}}x^{2/3}\right)^{2}\left(\frac{x^{2}}{3}-\frac{124\zeta\left(5\right)}{\pi^{5}}x^{3}\right)\\
 & =\frac{x^{2}\left(1-\frac{x^{2}}{8}\right)^{4}}{\left(x+\frac{x^{3}}{3}+\frac{2x^{5}}{5}\right)^{2}}\left(1+\frac{c_{1}}{2^{1/3}\pi^{2/3}}x^{2/3}\right)^{2}\left(1-\frac{3\cdot124\zeta\left(5\right)}{\pi^{5}}x\right)\\
 & =\frac{\left(1-\frac{x^{2}}{8}\right)^{4}}{\left(1+\frac{x^{2}}{3}+\frac{2x^{4}}{5}\right)^{2}}\left(1+\frac{c_1}{2^{1/3}\pi^{2/3}}x^{2/3}\right)^{2}\left(1-\frac{3\cdot124\zeta\left(5\right)}{\pi^{5}}x\right):=g(x).
\end{align*}

It suffices to show $g(x)\geq1$ for when $0\leq x\leq\tan^{-1}\left(1/3\right)\approx 0.32$,
since $b\geq 3$. To do this we consider the polynomial
\begin{align*}
P(x)= & \left(1+\frac{x^{2}}{3}+\frac{2x^{4}}{5}\right)^{2}-\left(1-\frac{x^{2}}{8}\right)^{4}\left(1+\frac{c_1}{2^{1/3}\pi^{2/3}}x^{2/3}\right)^{2}\left(1-\frac{3\cdot124\zeta\left(5\right)}{\pi^{5}}x\right). 
\end{align*}
We want to show that this polynomial satisfies $P(x)\leq0$ on $x\in\left(0,\tan^{-1}\left(\frac{1}{3}\right)\right)$.
Making  the substitution $x\mapsto x^{3}$ gives

\begin{equation}
Q(x)=\left(1+\frac{x^{6}}{3}+\frac{2x^{12}}{5}\right)^{2}-\left(1-\frac{x^{6}}{8}\right)^{4}\left(1+\frac{c_1}{2^{1/3}\pi^{2/3}}x^{2}\right)^{2}\left(1-\frac{3\cdot124\zeta\left(5\right)}{\pi^{5}}x^{3}\right)\label{polyineq-3}.
\end{equation}
We can show $Q$ is negative on $x\in\left(0,\left(\tan^{-1}\left(\frac{1}{3}\right)\right)^{1/3}\right)\subset\left(0,.686\right)$
by applying Siudeja's algorithm in Section \ref{Sec:Lemma3}. 
\end{proof}

\section{\label{sec:Lower-Triangle-Obtuse}Proof of Theorem \ref{thm:MainResult}:
Lower Bound for  Obtuse and Right Triangles}

Consider a triangle $\triangle_{a,b}$ with vertices $\left(0,0\right)$, $\left(1,0\right)$, and $\left(a,b\right)$
with sides of length $1,M=\sqrt{a^{2}+b^{2}}$ and $N=\sqrt{\left(a-1\right)^{2}+b^{2}}$.
Recall that by the discussion in Section \ref{sec:Triangle-Set-up},
to prove the desired bounds for all obtuse triangles we will use the following characterization
\[
\mathcal{T}_{obtuse}=\left\{ \left(a,b\right)\in\left[0,\frac{1}{2}\right]\times\left[0,\frac{1}{2}\right]\mid\left(a-\frac{1}{2}\right)^{2}+b^{2}\leq\frac{1}{4}\right\} .
\]
We will split the proof into three main cases. See Figure \ref{Obtuse-LowerBound-v2} for a picture of the regions for $(a,b)$.

\begin{figure}[h]
    \centering
    \includegraphics[width=0.4\textwidth]{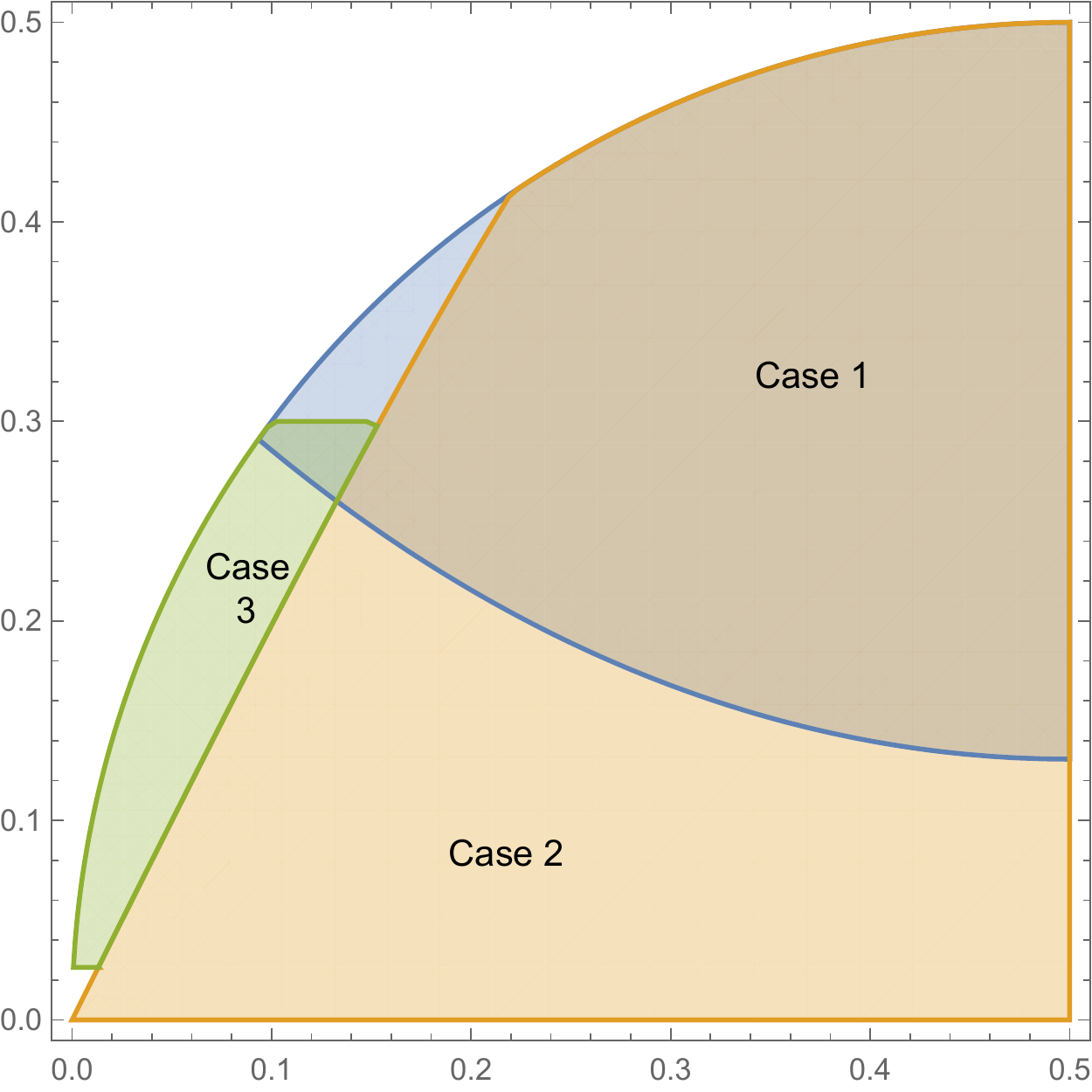}
    \caption{Cases for Lower bound for obtuse and right triangles }
    \label{Obtuse-LowerBound-v2}
\end{figure}

 For two of the propositions below we will use the following bound
by Freitas and Siudeja in \cite[Corollary 4.1]{Freitas-Siudeja-2010} 
\begin{equation}
\lambda_{1}\left(\triangle_{a,b}\right)\geq\pi^{2}\left(\frac{1}{d}+\frac{1}{h}\right)^{2},\label{eq:Bartek-Freitas-Eigen-Bound1}
\end{equation}
where $d$ is the diameter and $h$ is the height perpendicular to
its longest side. If $\left(a,b\right)\in\mathcal{T}_{obtuse}$ we
have that $d=1$ and $h=b$ so that 
\begin{equation}
\lambda_{1}\left(\triangle_{a,b}\right)\geq\pi^{2}\left(1+\frac{1}{b}\right)^{2}.\label{eq:Bartek-Freitas-Eigen-Bound2}
\end{equation}

\begin{prop}[Case 1]
\label{prop-Obtuse-Lower-Case1}If $\left(a,b\right)\in\mathcal{T}_{obtuse}$,
then 
\[
\frac{\pi^{2}}{24}< F\left(\bigtriangleup_{a,b}\right),
\]
for $$\frac{3}{2}-\frac{\sqrt{5}}{2}\sqrt{1+2a-2a^{2}}\leq b\leq\sqrt{a-a^{2}},$$
when $$0.0934\approx\frac{1}{6}\left(3-\sqrt{24\sqrt{15}-87}\right)\leq a\leq\frac{1}{2}.$$
\end{prop}

\begin{proof}
By Lemma \ref{lem:Test-Function-Approach-1} we have 
\[
T\left(\triangle_{a,b}\right)>\frac{b^{3}}{80\left[1-a+a^{2}+b^{2}\right]},
\]
where the inequality holds strictly since equality holds only for the equilateral triangle. Moreover (\ref{eq:Bartek-Freitas-Eigen-Bound2}) gives $\lambda_{1}\left(\triangle_{a,b}\right)\geq\pi^{2}\left(1+\frac{1}{b}\right)^{2}$,
so that 
\begin{align*}
F\left(\bigtriangleup_{a,b}\right) & \geq\frac{1}{b/2}\pi^{2}\left(1+\frac{1}{b}\right)^{2}\frac{b^{3}}{80\left[1-a+a^{2}+b^{2}\right]}\\
 & =\frac{\pi^{2}}{24}\frac{3}{5}\frac{\left(1+b\right)^{2}}{1-a+a^{2}+b^{2}}=:\frac{\pi^{2}}{24}f\left(a,b\right).
\end{align*}
A straightforward computation shows that $f\left(a,b\right)\geq1$
if 
\[
\frac{3}{2}-\frac{\sqrt{5}}{2}\sqrt{1+2a-2a^{2}}\leq b\leq\sqrt{a-a^{2}}
\]
which holds  for $$0.0934\approx\frac{1}{6}\left(3-\sqrt{24\sqrt{15}-87}\right)\leq a\leq\frac{1}{2}.$$
\end{proof}
\begin{prop}[Case 2]
\label{prop-Obtuse-Lower-Case2}If $\left(a,b\right)\in\mathcal{T}_{obtuse}$,
then 
\[
\frac{\pi^{2}}{24}\leq F\left(\bigtriangleup_{a,b}\right),
\]
for $0\leq b\leq\frac{2a\left(1-a\right)}{1-a+a^{2}}$.
\end{prop}

\begin{proof}
To bound the torsional rigidity we use the following test function
(similar to the one used in \cite[Section 5.1]{Vandenberg-Ferone-Nitsch-Trombetti-2019a}) in the variational
characterization of $T\left(\bigtriangleup_{a,b}\right)$ :
\[
v\left(x,y\right)=\begin{cases}
\frac{b^{2}x^{2}}{4a^{2}}-\left(y-\frac{bx}{2a}\right)^{2} & 0\leq x\leq a\\
\frac{b^{2}\left(1-x\right)^{2}}{4\left(1-a\right)^{2}}-\left(y-\frac{b\left(1-x\right)}{2\left(1-a\right)}\right)^{2} & a\leq x\leq1
\end{cases}.
\]
The triangle $\triangle_{a,b}$ is bounded by $0\leq y\leq\frac{b}{a}x,$
when $0\leq x\leq a$, by  $0\leq y\leq\frac{-b}{1-a}\left(x-1\right)$, 
when $a\leq x\leq1$, and by $y\geq0$, when $0\leq x\leq1$. It is then
clear that $v$ vanishes on the boundary. A straightforward computation
shows that 
\begin{align*}
T\left(\triangle_{a,b}\right) & \geq\frac{\left(\int_{\triangle_{a,b}}v\right)^{2}}{\int_{\triangle_{a,b}}\left|\nabla v\right|^{2}}=\frac{\frac{b^{6}}{576}}{\frac{b^{3}\left(a-a^{2}+b^{2}\right)}{12\left(1-a\right)a}}=\frac{\left(1-a\right)ab^{3}}{48\left(a-a^{2}+b^{2}\right)}.
\end{align*}

Combining this bound with $\left|\triangle_{a,b}\right|=\frac{b}{2}$
and (\ref{eq:Bartek-Freitas-Eigen-Bound2}) shows that $F\left(\triangle_{a,b}\right)\geq\frac{\pi^{2}}{24}f\left(a,b\right)$
with 
\begin{align*}
f\left(a,b\right) & =\frac{24}{\pi^{2}}\frac{1}{b/2}\pi^{2}\left(1+\frac{1}{b}\right)^{2}\left(\frac{\left(1-a\right)ab^{3}}{48\left(a-a^{2}+b^{2}\right)}\right)\\
 & =\frac{\left(1-a\right)a\left(1+b\right)^{2}}{\left(a-a^{2}+b^{2}\right)}.
\end{align*}
From this it is easy to see that $f\left(a,b\right)\geq1$ for 
\[
0\leq b\leq\frac{2a\left(1-a\right)}{1-a+a^{2}}.
\]
\end{proof}

Next, we prove the result for degenerating obtuse triangles that are close to right triangles. This will be the most difficult case in this section. The proof is similar to the proof of Proposition \ref{prop:Lower-Acute-b-high} in the acute case. We use sectors to give the appropriate lower bounds and then use the monotonicity results from Lemmas \ref{lem:Acute-High-1} and \ref{lem:Techinical-Lower-Mgeq3} to reduce to a lower bound for right triangles. 

\begin{prop}[Case 3]
\label{prop-Obtuse-Lower-Case3}If $\left(a,b\right)\in\mathcal{T}_{obtuse}$,
then 
\[
\frac{\pi^{2}}{24}< F\left(\bigtriangleup_{a,b}\right),
\]
for $\frac{2a\left(1-a\right)}{1-a+a^{2}}\leq b\leq\sqrt{a-a^{2}}$
and $0\leq a\leq 0.5, 0\leq b\leq 0.3$. 
\end{prop}

\begin{proof}
We first give an estimate for the torsional rigidity. This estimate
will be similar to the acute case given in Proposition \ref{prop:Lower-Acute-b-high}.

Let $\left(a,b\right)\in\mathcal{T}_{obtuse}$ and in particular consider
\[
\left(a,b\right)\in R_{Case3}=\left\{ \left(a,b\right)\mid\frac{2a\left(1-a\right)}{1-a+a^{2}}\leq b\leq\sqrt{a-a^{2}},0\leq a\leq 0.5, 0\leq b\leq 0.3\right\} .
\]
 For obtuse triangles one can see that $S\left(\beta,N\right)\subset\triangle_{a,b}$
where $N=\sqrt{\left(a-1\right)^{2}+b^{2}}$ is the side of middle
length and $\beta$ is the smallest angle on the lower right. This
containment can be checked using elementary geometry with a circle
centered at $\left(\frac{1}{2},0\right)$ of radius $N$ and noting
that such a circle must be contained in $\triangle_{a,b}$ since the
angle $\gamma$ between sides of length $M,N$ has an angle greater
than $\frac{\pi}{2}$. This implies that 
\begin{align*}
T\left(\triangle_{a,b}\right) & \geq T\left(S\left(\beta,N\right)\right)\\
 & =\frac{N^{4}}{16}\left(\tan\beta-\beta-\frac{128\beta^{4}}{\pi^{5}}\sum_{n=1,3,5}\frac{1}{n^{2}\left(n+\frac{2\beta}{\pi}\right)^{2}\left(n-\frac{2\beta}{\pi}\right)}\right).
\end{align*}

Recall that $\beta=\beta\left(\triangle_{a,b}\right)=\tan^{-1}\left(\frac{a}{1-b}\right)$.
We first like to give a bound on the range of $\beta$. Consider $$R^{\prime}=\left\{ \left(a,b\right)\mid\frac{2a\left(1-a\right)}{1-a+a^{2}}\leq b\leq\sqrt{a-a^{2}},0\leq a\leq\frac{1}{2}\left(1-\sqrt{8\sqrt{2}-11}\right)\right\}, $$
where $a=\frac{1}{2}\left(1-\sqrt{8\sqrt{2}-11}\right)$ is the intersecting
point for the curves $b=\frac{2a\left(1-a\right)}{1-a+a^{2}}$ and
$b=\frac{1}{2}\left(1-\sqrt{8\sqrt{2}-11}\right)$. Note that $R_{Case3}\subset R^{\prime}$.
For any fixed $b$, one can see that $\beta\left(\triangle_{a,b}\right)$
is an increasing function of $a$. For $\left(a,b\right)\in R^{\prime}$,
one  has that the angle $\beta\left(\triangle_{a,b}\right)=\tan^{-1}\left(\frac{a}{1-b}\right)$
is maximized when $a=\frac{1}{2}\left(1-\sqrt{8\sqrt{2}-11}\right)$
and $b=\frac{2a\left(1-a\right)}{1-a+a^{2}}$ so that $\beta\in\left[0, 0.489\right]\subset\left[0,\frac{\pi}{4}\right]$
in this region. Now recall that for $n\in\mathbb{N}$ we have $$\min_{0\leq\gamma\leq\frac{\pi}{4}}\left(n+\frac{2\gamma}{\pi}\right)^{2}\left(n-\frac{2\gamma}{\pi}\right)=n^{3},$$
so that 
\[
T\left(\triangle_{a,b}\right)\geq\frac{N^{4}}{16}\left(\tan\beta-\beta-\frac{128\beta^{4}}{\pi^{5}}\sum_{n=1,3,5}\frac{1}{n^{5}}\right).
\]
Since $$\sum_{n=1,3,5}\frac{1}{n^{5}}=\frac{31\zeta\left(5\right)}{32},$$ we can rewrite 
\begin{align*}
T\left(\triangle_{a,b}\right) & \geq\frac{N^{4}}{16}\left(\tan\beta-\beta-\frac{124\zeta\left(5\right)\beta^{4}}{\pi^{5}}\right).
\end{align*}

Using Lemma \ref{lem:Bartek-Lem1} for the obtuse case we have that
\[
\lambda_{1}\left(\triangle_{a,b}\right)\geq\frac{\beta}{b}j_{\pi/\beta}^{2},
\]
and using the same estimates as in the proof of Proposition \ref{prop:Lower-Acute-b-low}
we have that 
\[
\lambda_{1}\left(\triangle_{a,b}\right)>\frac{\beta}{b}\left(\frac{\pi}{\beta}+\frac{c_{1}}{2^{1/3}}\left(\frac{\pi}{\beta}\right)^{1/3}\right)^{2},
\]
with $c_{1}=2.338107$.  This gives  that 
\begin{align*}
F\left(\bigtriangleup_{a,b}\right) & >\frac{\pi^{2}}{24}\cdot\frac{24}{\pi^{2}}\frac{1}{b/2}\frac{\beta}{b}\left(\frac{\pi}{\beta}+\frac{c_{1}}{2^{1/3}}\frac{\pi^{1/3}}{\beta^{1/3}}\right)^{2}\frac{N^{4}}{16}\left(\tan\beta-\beta-\frac{124\zeta\left(5\right)\beta^{4}}{\pi^{5}}\right).
\end{align*}

Recall that for any fixed $b$ the angle $\beta\left(\triangle_{a,b}\right)$
is an increasing function of $a$. This  means for any fixed $b$,
the angle $\beta\left(\triangle_{a,b}\right)$ is minimized by $a_{b}=\frac{1}{2}-\sqrt{\frac{1}{4}-b^{2}}$,
which falls on the curve $\left(a-\frac{1}{2}\right)^{2}+b^{2}=\frac{1}{4}$
that represents the right triangles. By Lemma \ref{lem:Acute-High-1},
the map
\[
\beta\mapsto\left(\tan\beta-\beta-\frac{124\zeta\left(5\right)\beta^{4}}{\pi^{5}}\right)\beta\left(\frac{\pi}{\beta}+\frac{c_{1}}{2^{1/3}}\left(\frac{\pi}{\beta}\right)^{1/3}\right)^{2}
\]
is increasing for $\beta\in\left(0,0.7\right)$, so that since $\beta\left(\triangle_{a,b}\right)\geq\beta\left(\triangle_{a_{b},b}\right)=:\beta_{b}$
we have
\begin{align*}
 & F\left(\bigtriangleup_{a,b}\right)\\
 & >\frac{\pi^{2}}{24}\cdot\frac{24}{\pi^{2}}\frac{1}{b/2}\frac{\beta_{b}}{b}\left(\frac{\pi}{\beta_{b}}+\frac{c_{1}}{2^{1/3}}\frac{\pi^{1/3}}{\beta_{b}^{1/3}}\right)^{2}\frac{N^{4}}{16}\left(\tan\beta_{b}-\beta_{b}-\frac{124\zeta\left(5\right)\beta_{b}^{4}}{\pi^{5}}\right)\\
 & =\frac{\pi^{2}}{24}\cdot\frac{3}{b^{2}\beta_{b}}\left(\left(1-a_{b}\right)^{2}+b^{2}\right)^{2}\left(1+\frac{c_{1}}{2^{1/3}}\frac{\beta_{b}^{2/3}}{\pi^{2/3}}\right)^{2}\left(\frac{1}{\frac{1-a_{b}}{b}}-\tan^{-1}\left(\frac{1}{\frac{1-a_{b}}{b}}\right)-\frac{124\zeta(5)\left(\tan^{-1}\left(\frac{1}{\frac{1-a_{b}}{b}}\right)\right)^{4}}{\pi^{4}}\right)\\
 & =\frac{\pi^{2}}{24}\cdot\frac{4\left(\left(1-a_{b}\right)+\frac{b^{2}}{\left(1-a_{b}\right)}\right)^{2}}{\left(1+\frac{x_{b}}{\sqrt{x_{b}^{2}+1}}\right)^{2}}\\
 & \times\frac{3}{4}\frac{x_{b}^{2}}{\beta_{b}}\left(1+\frac{x_{b}}{\sqrt{x_{b}^{2}+1}}\right)^{2}\left(1+\frac{c_{1}}{2^{1/3}}\frac{\beta_{b}^{2/3}}{\pi^{2/3}}\right)^{2}\left(\frac{1}{x_{b}}-\tan^{-1}\left(\frac{1}{x_{b}}\right)-\frac{124\zeta(5)\left(\tan^{-1}\left(\frac{1}{x_{b}}\right)\right)^{4}}{\pi^{4}}\right),
\end{align*}
where $x_{b}=\frac{1-a_{b}}{b}$. Using Lemma \ref{lem:Techinical-Lower-Mgeq3}
shows that the factor 
\[
f\left(x_{b}\right)=\frac{3}{4}\frac{x_{b}^{2}}{\beta_{b}}\left(1+\frac{x_{b}}{\sqrt{x_{b}^{2}+1}}\right)^{2}\left(1+\frac{c}{2^{1/3}}\frac{\beta_{b}^{2/3}}{\pi^{2/3}}\right)^{2}\left(\frac{1}{x_{b}}-\tan^{-1}\left(\frac{1}{x_{b}}\right)-\frac{124\zeta(5)\left(\tan^{-1}\left(\frac{1}{x_{b}}\right)\right)^{4}}{\pi^{4}}\right)
\]
satisfies $f\left(x_{b}\right)\geq1$ for $x_{b}\geq3$. Noting that
$$x_{b}=\frac{1-a_{b}}{b}=\frac{1-\left(\frac{1}{2}-\sqrt{\frac{1}{4}-b^{2}}\right)}{b}\geq3,$$
whenever $0\leq b\leq 0.3$ gives  the desired bound for the factor
$f\left(x_{b}\right)$. 

Finally, we'd like to show the leftover factor term above is also greater
than $1$. To see this we simplify 
\begin{align*}
\frac{4\left(\left(1-a_{b}\right)+\frac{b^{2}}{\left(1-a_{b}\right)}\right)^{2}}{\left(1+\frac{x_{b}}{\sqrt{x_{b}^{2}+1}}\right)^{2}} & =\frac{16\left(1+\sqrt{1-4b^{2}}\right)}{\left(\sqrt{2}+\sqrt{2}\sqrt{1-4b^{2}}+2\sqrt{1+\sqrt{1-4b^{2}}}\right)^{2}}
\end{align*}
and an elementary computation shows that this term is greater than
$1$ for $0\leq b\leq 0.5$. Putting these bounds together shows the
desired result for $\left(a,b\right)\in R_{Case3}$. 
\end{proof}



\section{Proof of Theorem \ref{thm:MainResult}: Upper Bound for triangles and tangential quadrilaterals}\label{upperbounds}

The upper bound for triangles and tangential quadrilaterals will follow by the following proposition. Rectangles that are not squares are not tangential quadrilaterals, hence it will be treated separately in Section \ref{rectanglecase}. 
\begin{prop}
\label{thm:Upper-Bound-all}For all triangles $\bigtriangleup$ we have 
\begin{equation}\label{upper-for-triangles}
F\left(\bigtriangleup\right)\leq\frac{2\pi^{2}}{27}<\frac{\pi^{2}}{12}.
\end{equation}
Moreover, for any tangential quadrilateral $Q$ we have  
\begin{equation}\label{upper-for-rectangles} 
F\left(Q\right)<\frac{\pi^{2}}{12}.
\end{equation}
\end{prop}

\begin{proof}
Let $P(D)$ denote the perimeter of a convex domain $D$. By a result of Makai \cite{Makai-1962} (see also a more general version given in \cite{Fragala-Gazzola-Lamboley-2013}), we have  that for all planar convex
domains
\begin{equation}
\frac{T\left(D\right)P\left(D\right)^{2}}{\left|D\right|^{3}}<\frac{2}{3},\label{eq:Makai-Tor-Per-Area}
\end{equation}
where this upper bound is achieved by thinning isosceles triangles.
Moreover, a result of Siudeja \cite{Siudeja-2007} (see also \cite{Solynin-Zalgaller-2010} for a different proof) gives  that
\begin{equation}
\frac{\lambda_{1}\left(\triangle\right)\left|\triangle\right|^{2}}{P\left(\triangle\right)^{2}}\leq\frac{\pi^{2}}{9},\label{eq:Bartek-Eigen-Per-Area}
\end{equation}
where the upper bound is sharp for the equilateral triangle. Using
$(\ref{eq:Makai-Tor-Per-Area})$ and $(\ref{eq:Bartek-Eigen-Per-Area})$
we have
\begin{align*}
\frac{\lambda_{1}\left(\triangle\right)T\left(\triangle\right)}{\left|\triangle\right|} & =\frac{\lambda_{1}\left(\triangle\right)\left|\triangle\right|^{2}}{P\left(\triangle\right)^{2}}\frac{T\left(\triangle\right)P\left(\triangle\right)^{2}}{\left|\triangle\right|^{3}}\leq \frac{\pi^{2}}{9}\frac{2}{3}=\frac{2\pi^{2}}{27}<\frac{\pi^{2}}{12},
\end{align*}
as needed. 

Moreover, for any tangential quadrilateral $Q$, we have
\begin{equation}
\frac{\lambda_{1}\left(Q\right)\left|Q\right|^{2}}{P\left(Q\right)^{2}}\leq\frac{\pi^{2}}{8},\label{eq:Bartek-Eigen-Per-Area-1}
\end{equation}
by a result of Solynin and Zalgaller in \cite[Corollary 3]{Solynin-Zalgaller-2010}. Combining \eqref{eq:Bartek-Eigen-Per-Area-1} with Makai's inequality
it follows that 
\begin{align*}
\frac{\lambda_{1}\left(Q\right)T\left(Q\right)}{\left|Q\right|} & =\frac{\lambda_{1}\left(Q\right)\left|Q\right|^{2}}{P\left(Q\right)^{2}}\frac{T\left(Q\right)P\left(Q\right)^{2}}{\left|Q\right|^{3}}<\frac{\pi^{2}}{8}\frac{2}{3}=\frac{\pi^{2}}{12}.
\end{align*}

\end{proof}

We also give the following upper bound which gives the sharpness of the lower bound of Theorem \ref{thm:MainResult} for thinning triangles and tangential polygons. Recall that $R_D$ denotes the inradius of the domain $D$.

\begin{prop}\label{Prop:upper-thinning}
Suppose $D$ is a convex planar domain satisfying $\frac{1}{2}P\left(D\right)R_D=\left|D\right|$
(such as triangles and tangential polygons).  Then 
\[
F(D)\leq\frac{\pi^{2}}{24}\left(1+2\sqrt{\pi}\frac{\left|D\right|^{1/2}}{P\left(D\right)}\right)^{2}.
\]
Moreover,  if $D_{n}$ is a sequence of such domains so that $\frac{\left|D_{n}\right|^{1/2}}{P(D_{n})}\to0$, 
then 
\[
\limsup_{n\to\infty}F\left(D_{n}\right)\leq\frac{\pi^{2}}{24}.
\]
\end{prop}

\begin{proof}
By \cite[Corollary 1]{Ftouhi-2021}, we have the following bound for convex domains satisfying
$\frac{1}{2}P\left(D\right)R_D=\left|D\right|$: 
\[
\lambda_{1}(D)\leq\frac{\pi^{2}}{16}\left(\frac{P\left(D\right)}{\left|D\right|}+2\sqrt{\frac{\pi}{\left|D\right|}}\right)^{2}.
\]
Using Makai's inequality $(\ref{eq:Makai-Tor-Per-Area})$ it follows that 
\begin{align*}
\frac{\lambda_{1}\left(D\right)T\left(D\right)}{\left|D\right|} & \leq\frac{\pi^{2}}{16}\frac{2}{3}\frac{\left|D\right|^{2}}{P\left(D\right)^{2}}\left(\frac{P\left(D\right)}{\left|D\right|}+2\sqrt{\frac{\pi}{\left|D\right|}}\right)^{2}\\
 & =\frac{\pi^{2}}{24}\left(1+2\sqrt{\pi}\frac{\left|D\right|^{1/2}}{P\left(D\right)}\right)^{2},
\end{align*}
as required. 
\end{proof}

\section{Proof of Theorem \ref{thm:MainResult} for all triangles}\label{CollectAll} 
\begin{proof}[Proof of Theorem \ref{thm:MainResult}]
We collect all the results to prove our main theorem on triangles. Recall that
for the acute case/right triangle case we use the following representation,
\[
\mathcal{T}_{acute}^{\prime}=\left\{ \left(a,b\right):a^{2}+b^{2}\geq1,0\leq a\leq\frac{1}{2}\right\} 
\]
while for the obtuse case we use the following representation
\[
\mathcal{T}_{obtuse}=\left\{ \left(a,b\right)\in\left[0,\frac{1}{2}\right]\times\left[0,\frac{1}{2}\right]\mid\left(a-\frac{1}{2}\right)^{2}+b^{2}\leq\frac{1}{4}\right\} .
\]

\textbf{Part (a):} Acute Lower bound: Let $\left(a,b\right)\in\mathcal{T}_{acute}^{\prime}$.
The lower bound follows by Proposition \ref{prop:Lower-Acute-b-low}
for $\frac{\sqrt{3}}{2}\leq b\leq4$ and by Proposition \ref{prop:Lower-Acute-b-high}
for $b\geq3$. 

\textbf{Part (b):} Obtuse Lower bound: Let $\left(a,b\right)\in\mathcal{T}_{obtuse}$.
The lower bound follows by Proposition \ref{prop-Obtuse-Lower-Case1}
for $$\frac{3}{2}-\frac{\sqrt{5}}{2}\sqrt{1+2a-2a^{2}}\leq b\leq\sqrt{a-a^{2}},$$
when $$0.0934\approx\frac{1}{6}\left(3-\sqrt{24\sqrt{15}-87}\right)\leq a\leq\frac{1}{2}.$$
The lower bound also follows by Proposition \ref{prop-Obtuse-Lower-Case2}
for $0\leq b\leq\frac{2a\left(1-a\right)}{1-a+a^{2}}$ and by Proposition
\ref{prop-Obtuse-Lower-Case3} for $$\frac{2a\left(1-a\right)}{1-a+a^{2}}\leq b\leq\sqrt{a-a^{2}}, \quad 
\text{and}\quad 0\leq a\leq 0.5, 0\leq b\leq 0.3.$$ All these regions combined make
up $\mathcal{T}_{obtuse}$ (see Figure \ref{Obtuse-LowerBound-v2}). 

\textbf{Part (c):} Upper Bound: This follows by Proposition \ref{thm:Upper-Bound-all}.

\textbf{Part (d):} Sharpness of the lower bound for triangles: 

Let $\triangle_n$ be a sequence of triangles collapsing down to an interval. Then by a result of \cite[Proposition 5.2]{Berg-Buttazzo-Pratelli-2021} or \cite[Theorem 4.4]{Briani-Buttazzo-Prinari-2022} (see also \cite{Borisov-Freitas-2013,Borisov-Freitas-2010}) shows that $\lim_{n\to \infty}F\left(\triangle_{n}\right)=\frac{\pi^{2}}{24}$. The upper bound in Proposition \ref{Prop:upper-thinning} also shows this. 
\end{proof}




\section{Upper and lower bounds for rectangles}\label{rectanglecase}

For rectangles we can obtain the conjectured  bounds which will follow from the following stronger monotonicity  property of $F$.

\begin{thm}
\label{thm:Main-Rec} Let $R_{a,b}=\left(-a,a\right)\times\left(-b,b\right)$
be a rectangle domain. For all $a,b>0$, we have 
\begin{equation}
\frac{\pi^2}{24}<F\left(R_{1,1}\right)\leq F\left(R_{a,b}\right)<\frac{\pi^{2}}{12}\label{eq:thm:Main-Rec-inq}
\end{equation}
where $\lim_{a\to\infty}F\left(R_{a,1}\right)=\frac{\pi^{2}}{12}$.
Moreover, the function $F\left(R_{a,1}\right)$ is increasing for
$a\geq1$. 
\end{thm}

We first  recall the well-known infinite series formula for $T(R_{a,b})$.  For the interval $I_{a}=(-a,a)$
the eigenfunctions are given by  $$
\varphi_{n}\left(x\right)=\frac{1}{\sqrt{a}}\sin\left(\frac{n\pi}{2a}\left(x+a\right)\right)
$$ with eigenvalues 
$
\lambda_{n}=\left(\frac{n\pi}{2a}\right)^{2}.
$
Let $p_{(-a,a)}(x,y,t)$ denotes the transition densities for killed Brownian
motion $B_{t}$ on the interval $I_{a}=(-a,a)$, that is, the Dirichlet heat kernel for $\frac{1}{2}\Delta$ in the interval $I_a$.   By the eigenfunction expansion we have 
\begin{align*}
p_{(-a,a)}(x,y,t) & =\sum_{n=1}^{\infty}e^{-\lambda_{n}t/2}\varphi_{n}(x)\varphi_{n}(y)\\
 & =\sum_{n=1}^{\infty}\frac{1}{a}\exp\left(-\left(\frac{n\pi}{2a}\right)^{2}\frac{t}{2}\right)\sin\left(\frac{n\pi}{2a}\left(x+a\right)\right)\sin\left(\frac{n\pi}{2a}\left(y+a\right)\right). 
\end{align*}
From this, the product structure of the rectangles $R_{a,b}=\left(-a,a\right)\times\left(-b,b\right)$ and independence of the components of 
the Brownian motion, it follows that 
the torsion function for $R_{a,b}$ is given by 
\begin{align*}
u_{a,b}(x,y) & =\frac{1}{2}\mathbb{E}_{(x,y)}\left[\tau_{(-a,a)\times(-b,b)}\right] \\
 & =\frac{4^{3}a^{2}}{\pi^{4}}\sum_{n=0}^{\infty}\frac{\sin\left(\frac{(2n+1)\pi}{2a}\left(x+a\right)\right)}{\left(2n+1\right)^{3}}\sum_{m=0}^{\infty}\frac{\sin\left(\frac{(2m+1)\pi}{2b}\left(y+b\right)\right)}{\left(2m+1\right)}\frac{1}{\left(1+\frac{a^{2}\left(2m+1\right)^{2}}{b^{2}\left(2n+1\right)^{2}}\right)}.
\end{align*}
Integrating this the torsional rigidity is given by 
\begin{equation}\label{ExpTorRec}
T(R_{a,b})=\frac{4^{5}a^{3}b^{3}}{\pi^{6}}\sum_{n=0}^{\infty}\sum_{m=0}^{\infty}\frac{1}{\left(b^{2}\left(2n+1\right)^{4}\left(2m+1\right)^{2}+a^{2}\left(2m+1\right)^{4}\left(2n+1\right)^{2}\right)}.
\end{equation}
This formula can be found in various places, see for example  \cite[p. 108]{Polya-Szego-1951}. Hence the P\'olya functional for $R_{a, b}$ is given by 
\begin{align}\label{ExpPolRec}
F\left(R_{a,b}\right) 
 & =\frac{4^{3}}{\pi^{4}}\sum_{n=0}^{\infty}\sum_{m=0}^{\infty}\frac{\left(a^{2}+b^{2}\right)}{\left(b^{2}\left(2n+1\right)^{4}\left(2m+1\right)^{2}+a^{2}\left(2m+1\right)^{4}\left(2n+1\right)^{2}\right)}.
\end{align}

\begin{proof}[Proof of Theorem \ref{thm:Main-Rec}] 
To prove the monotonicity property observe that since $F$ is scale invariant and 
$
R_{a,b}=aR_{1,\frac{b}{a}}, 
$
it suffices to consider 
\[
F\left(R_{x,1}\right)=\frac{4^{3}}{\pi^{4}}\sum_{n=0}^{\infty}\sum_{m=0}^{\infty}\frac{\left(1+x^{2}\right)}{\left(\left(2n+1\right)^{4}\left(2m+1\right)^{2}+x^{2}\left(2m+1\right)^{4}\left(2n+1\right)^{2}\right)},
\]
for $x\geq1$. Setting $f(x)=\frac{\pi^4}{4^3}F\left(R_{x,1}\right)$, our goal is to  show that $f$ is increasing. First rewrite
\begin{align*}
 & f\left(x\right)\\
 & =\left(\sum_{n=0}^{\infty}\sum_{m:n=m}+\sum_{n=0}^{\infty}\sum_{m:n>m}+\sum_{n=0}^{\infty}\sum_{m:n<m}\right)\frac{1+x^{2}}{\left(\left(2n+1\right)^{4}\left(2m+1\right)^{2}+x^{2}\left(2m+1\right)^{4}\left(2n+1\right)^{2}\right)}\\
 & =\sum_{n=0}^{\infty}\frac{1+x^{2}}{\left(\left(2n+1\right)^{6}+x^{2}\left(2n+1\right)^{6}\right)}+\sum_{n=0}^{\infty}\sum_{m:n>m}\left[\frac{1+x^{2}}{\left(\left(2n+1\right)^{4}\left(2m+1\right)^{2}+x^{2}\left(2m+1\right)^{4}\left(2n+1\right)^{2}\right)}\right]\\
 & +\sum_{n=0}^{\infty}\sum_{m:n<m}\left[\frac{1+x^{2}}{\left(\left(2n+1\right)^{4}\left(2m+1\right)^{2}+x^{2}\left(2m+1\right)^{4}\left(2n+1\right)^{2}\right)}\right]\\
 & =\sum_{n=0}^{\infty}\frac{1}{\left(2n+1\right)^{6}}+\sum_{n=0}^{\infty}\sum_{m:n>m}\left[\frac{1+x^{2}}{\left(\left(2n+1\right)^{4}\left(2m+1\right)^{2}+x^{2}\left(2m+1\right)^{4}\left(2n+1\right)^{2}\right)}\right]\\
 & +\sum_{n=0}^{\infty}\sum_{m:n>m}\left[\frac{1+x^{2}}{\left(\left(2m+1\right)^{4}\left(2n+1\right)^{2}+x^{2}\left(2n+1\right)^{4}\left(2m+1\right)^{2}\right)}\right]\\
 & =\sum_{n=0}^{\infty}\frac{1}{\left(2n+1\right)^{6}}+\sum_{n=0}^{\infty}\sum_{m:n>m}\left[g_{\alpha,\beta}\left(x\right)\right],
\end{align*}
where $\alpha=\alpha_{n,m}=\left(2n+1\right)^{4}\left(2m+1\right)^{2}$
, $\beta=\beta_{n,m}=\left(2m+1\right)^{4}\left(2n+1\right)^{2}$ and
\[
g_{\alpha,\beta}\left(x\right)=\frac{1+x^{2}}{\left(\alpha+x^{2}\beta\right)}+\frac{1+x^{2}}{\left(\beta+x^{2}\alpha\right)}.
\]
Thus to show that $f(x)$ is increasing for $x\geq 1$, it suffices to show that each $g_{\alpha,\beta}$
is increasing for $x\geq 1$. To this end, we compute the derivative of $g_{\alpha,\beta}$ to obtain that 
\[
g_{\alpha,\beta}^{\prime}\left(x\right)=\frac{2\left(\alpha-\beta\right)^{2}\left(\alpha+\beta\right)x\left(x^{4}-1\right)}{\left(\beta+\alpha x^{2}\right)^{2}\left(\alpha+\beta x^{2}\right)^{2}}
\]
Clearly this quantity is nonnegative for $x\geq 1$.  Hence $g_{\alpha,\beta}$ is increasing. Thus, $F\left(R_{1,1}\right)\leq F\left(R_{a,b}\right)$. Moreover, taking $n=0,m=0$ in \eqref{ExpPolRec} it follows that $F\left(R_{a,b}\right)\geq \frac{64}{\pi^4} \geq \frac{\pi^2}{24}$, which gives the desired lower bound.

Finally note that for $x\geq1$, we have
\[
\frac{\left(1+x^{2}\right)}{\left(\left(2n+1\right)^{4}\left(2m+1\right)^{2}+x^{2}\left(2m+1\right)^{4}\left(2n+1\right)^{2}\right)}\leq\frac{2}{\left(2m+1\right)^{4}\left(2n+1\right)^{2}}
\]
and $\sum_{n=0}^{\infty}\sum_{m=0}^{\infty}\frac{2}{\left(2m+1\right)^{4}\left(2n+1\right)^{2}}<\infty$. 
Thus by the Dominated Convergence Theorem,
\begin{align*}
\lim_{x\to\infty}F\left(R_{x,1}\right) & =\frac{4^{3}}{\pi^{4}}\sum_{n=0}^{\infty}\sum_{m=0}^{\infty}\lim_{x\to\infty}\frac{\left(1+x^{2}\right)}{\left(\left(2n+1\right)^{4}\left(2m+1\right)^{2}+x^{2}\left(2m+1\right)^{4}\left(2n+1\right)^{2}\right)}\\
 & =\frac{4^{3}}{\pi^{4}}\sum_{n=0}^{\infty}\frac{1}{\left(2n+1\right)^{2}}\sum_{m=0}^{\infty}\frac{1}{\left(2m+1\right)^{4}}\\
 & =\frac{4^{3}}{\pi^{4}}\left(\frac{\pi^{2}}{8}\right)\left(\frac{\pi^{4}}{96}\right)=\frac{\pi^{2}}{12}, 
\end{align*}
which competes the proof of Theorem \ref{thm:Main-Rec}
\end{proof}

We end this section with the following remarks. 
\begin{rem}
Bounds for the torsional rigidity of a rectangle are known. For example, a lower bound for rectangles was known to  P\'olya-Szego \cite[p. 99, Eq. 4]{Polya-Szego-1951},  while an upper bound is given in \cite[pp. 46]{Fleeman-Simanek-2019},
\[
\frac{a^{3}b^{3}}{\left(a^{2}+b^{2}\right)}\leq T(R_{a,b})\leq\frac{4a^{3}b^{3}}{3\left(a^{2}+b^{2}\right)}.
\]
Combining these with the exact value for the principal eigenvalue of a rectangle gives the following bounds for the P\'olya functional
\[
\frac{\pi^{2}}{16}\leq F(R_{a,b})\leq\frac{\pi^{2}}{12},
\]
which would immediately prove Conjecture \ref{Conjecture-convex}   for rectangles. Our Theorem gives more information on the behavior of the function $F(R_{x,1})$ than just the inequality. It is also interesting to note that  by taking (as before) $n=0$ and $m=0$ in the series \eqref{ExpTorRec} we obtain 
\begin{equation}\label{lowerSer}
\frac{4^5a^{3}b^{3}}{\pi^6\left(a^{2}+b^{2}\right)}\leq T(R_{a,b}).  
\end{equation}
 Clearly better bounds can be obtained by using more terms in the series. The bound \eqref{lowerSer} was already noticed by P\'olya-Szego, see  \cite[p. 108]{Polya-Szego-1951}.
\end{rem}

\begin{remark}[A related functional for rectangles] Theorem \ref{thm:Main-Rec} is related to a conjecture in  \cite[Remark 5.5]{Banuelos-Mariano-Wang-2020} for the 
functional 
\begin{equation}\label{Max}
G\left(D\right)=\lambda_{1}\left(D\right)\left\Vert u_{D}\right\Vert _{\infty},
\end{equation}
where $u_{D}\left(x\right)$
is the torsion function in \eqref{TorionFunc}.   The conjecture in \cite[Remark 5.5]{Banuelos-Mariano-Wang-2020} states that over all rectangles
$R_{a,1}$, the functional $G\left(R_{a,1}\right)$ is maximized by the square
$S$. In  fact, the conjecture in \cite{Banuelos-Mariano-Wang-2020} is stated for all dimension and expectations of $p$ powers of the exit time of Brownian motion for $p\geq 1$.  The conjecture remains open although it is not difficult to verify it when $p=1$  for rectangles that are sufficiently long.  More precisely, with our notation $R_{a,b}=\left(-a,a\right)\times\left(-b,b\right)$, the conjecture in 
\cite[eq. (5.22)]{Banuelos-Mariano-Wang-2020} for the functional \eqref{Max} is equivalent to proving that $G(R_{a,1})\leq G(R_{1,1})$ for $a\geq 1$.  This is the simplest case of the conjecture with $d=2$ and $p=1$. Since for a rectangle $\Vert u_{R_{a,b}}\Vert _{\infty}=u_{R_{a,b}}(0, 0)= \frac{1}{2}\mathbb{E}_{(0,0)}\left[\tau_{R_{a,b}}\right]$, the conjecture follows for $a>2.38$  by estimating this quantity with the corresponding quantity for the infinite strip $(-\infty, \infty)\times (-1, 1)$. 
To see this, we can use the fact that $\mathbb{E}_{\left(0,0\right)}\left[\tau_{R_{\frac{\sqrt{2}}{2},\frac{\sqrt{2}}{2}}}\right]\approx 0.294685$
(see \cite[Equation (25)]{Markowsky-2011}) and $\lambda_{1}\left(R_{\frac{\sqrt{2}}{2},\frac{\sqrt{2}}{2}}\right)=\pi^{2}$
so that $G\left(R_{1,1}\right)=G\left(R_{\frac{\sqrt{2}}{2},\frac{\sqrt{2}}{2}}\right)\geq\frac{1}{2}\pi^{2}\left(0.294\right)\geq1.45$.
By domain monotonicity, $\mathbb{E}_{\left(0,0\right)}\left[\tau_{R_{a,1}}\right]\leq\mathbb{E}_{\left(0,0\right)}\left[\tau_{R_{\infty,1}}\right]=1$. Thus, 
 
\[
G\left(R_{a,1}\right)=\lambda_{1}\left(R_{a,1}\right)\frac{1}{2}\mathbb{E}_{\left(0,0\right)}\left[\tau_{R_{a,1}}\right]\leq\frac{\pi^{2}}{8}\left(1+\frac{1}{a^{2}}\right).
\]
On the other hand since 
$$\frac{\pi^{2}}{8}\left(1+\frac{1}{a^{2}}\right)a^{2}\leq1.45, \quad \text{for}\quad 
a\geq\sqrt{\frac{5\pi^2}{58-5\pi^{2}}},$$ we have $$G\left(R_{a,1}\right)\leq G\left(R_{1,1}\right), \quad \text{for}\quad a\geq2.38.$$ 

The functional $G(D)$ has been extensively studied in the literature.  For more discussion on this functional, we refer to \cite{Banuelos-Mariano-Wang-2020} and the many references contained therein.  
\end{remark}


\section{\label{sec:Polynomial-Inequality}Algorithm for proving polynomial
inequalities}

We use a simple idea for proving polynomial inequalities introduced by Siudeja in \cite[Section 5]{Siudeja-IU-2010}. We refer to Section 5 of \cite{Siudeja-IU-2010} for a full description of the algorithm and
a proof for the validity of the algorithm.  The algorithm can be used for proving
polynomial inequalities of the form $Q\left(x,y\right)\leq0$ over
a rectangle $\left(0,a\right)\times\left(0,b\right)$. 

In our paper, we will only need to prove polynomial inequalities of
the form 
\[
P\left(x\right):=a_{0}+a_{1}x+\cdots+a_{n}x^{n}\leq0
\]
for  $x$ in an interval of the form $\left(0,a\right)$ for some $a>0$.  The idea is to reduce terms with positive
coefficients to lower powers using 
\begin{align*}
a_{i}x^{i} & \leq a_{i}ax^{i-1}\text{ when }a_{i}>0,\\
a_{i}x^{i} & \leq a_{i}a^{-1}x^{i+1}\text{ when }a_{i}<0.
\end{align*}
Eventually this should reduce to a polynomial with only negative coefficients.
If this doesn't work, we subdivide the interval $\left(0,a\right)$
into two subintervals and repeat the same process on each of the two
subintervals. We  keep subdividing until the result is proven.

The computations to prove the polynomial inequalities from Lemmas \ref{lem:technical-func-ineq-1}, \ref{lem:Acute-High-1} and \ref{lem:Techinical-Lower-Mgeq3} can be done by hand but are long and very tedious. We use 
Siudeja's algorithm in Mathematica$\copyright$ to prove the polynomial inequalities from these Lemmas. For convenience, we include the Mathematica code from \cite[Section 5]{Siudeja-IU-2010} and \cite[Section 5.2]{Siudeja-HotSpots-2015} that is also used in this paper.

\vskip.5cm

\hrule
\vskip.1cm
\noindent\textbf{Algorithm 1.} Algorithm for proving polynomial inequality $P\left(x,y\right)\leq0$ over
a rectangle $\left(0,dx \right)\times\left(0,dy \right)$
\vskip.1cm
\hrule

\vskip.3cm

\noindent CumFun[f$\_$, l$\_$] $:=$ \textbf{Rest}[\textbf{FoldList}[f, 0, l]]; \\
PolyNeg[P$\_$, \{x$\_$, y$\_$\}, \{dx$\_$, dy$\_$\}] $:=$ \\
((\textbf{Fold}[CumFun[\textbf{Min}[\#1,0]/dy+\#2\&, \textbf{Map}[\textbf{Max}[\#1,0]\&,\#1]dx+\#2]\&,0,\\
\textbf{Reverse}[\textbf{CoefficientList}[P,\{x, y\}]]]//\textbf{Max}) $<=0$);

\vskip.3cm
\hrule

\subsection{Polynomial inequality in Lemma \ref{lem:technical-func-ineq-1} }\label{Sec:Lemma1}

Consider the polynomial $P_2(x)$ defined in \eqref{polyineq-1}. We can show that $P_2(x)\leq 0 $  for $x\in\left(0,.285\right)$ is true by running the algorithm 
PolyNeg[P2[x], \{x, y\}, \{.285, 1\}].

\subsection{Polynomial inequality in Lemma \ref{lem:Acute-High-1} }\label{Sec:Lemma2}

Consider the polynomial $-P_{1}^{\prime}(x)$ defined from  the derivative of $P_1$ given in  \eqref{polyineq-2}. To show that $-P_{1}^{\prime}(x)\leq0$
we use the  algorithm on $\left(0,\frac{0.888}{2}\right)$ and $\left(\frac{0.888}{2}, 0.888\right)$.

In particular, we use the algorithm PolyNeg[$-P_{1}^{\prime}$[x] , \{x, y\}, \{0.888/2, 1\}] to show the desired inequality is true on  $\left(0,\frac{0.888}{2}\right)$. To show that $-P_{1}^{\prime}(x)\leq0$  on $\left(\frac{0.888}{2}, 0.888\right)$ we define the polynomial 
\[
P_{2}\left(x\right)=-P_{1}^{\prime}\left(x+\frac{0.888}{2}\right)
\]
and run the algorithm PolyNeg[P2[x], \{x, 
  y\}, \{0.888/2, 1\}]. 

\subsection{Polynomial inequality in Lemma \ref{lem:Techinical-Lower-Mgeq3} }\label{Sec:Lemma3}

Consider the polynomial $Q(x)$ defined in \eqref{polyineq-3}. We can show that $Q(x)\leq 0 $  for $x\in\left(0, .686\right)$ is true by running the algorithm PolyNeg[Q[x], \{x, y\}, \{0.686, 1\}].

\begin{acknowledgement}
 Part of this research was conducted during an extended visit by P. Mariano to Purdue University during the fall semester of the 2023-2024 academic year. He would like to thank the department of mathematics at Purdue for the hospitality which allowed for many fruitful conversations with several colleagues there and visitors to the department. 
\end{acknowledgement}

\newpage
\bibliographystyle{plain}	

\bibliography{MainRef}

\end{document}